\documentclass[11pt]{amsart}

\usepackage[utf8]{inputenc}
\usepackage[T1]{fontenc}

\usepackage{amsmath,amssymb,amsfonts,textcomp,amsthm,xifthen,graphicx,color}

\usepackage{enumerate}
\usepackage{enumitem}

\usepackage{fullpage}

\usepackage[utf8]{inputenc}

\usepackage[pdftex,
            pdfauthor={Thomas F\"uhrer and Michael Karkulik},
            pdftitle={Least-squares finite elements for distributed optimal control problems},
            ]{hyperref}

\usepackage{pgfplots}
\usepgfplotslibrary{groupplots}
\pgfplotsset{compat=newest}
\usepgfplotslibrary{external}
\usepgfplotslibrary{colorbrewer}
\newtheorem{theorem}{Theorem}
\newtheorem{lemma}[theorem]{Lemma}
\newtheorem{corollary}[theorem]{Corollary}
\newtheorem{proposition}[theorem]{Proposition}
\newtheorem{remark}[theorem]{Remark}

\newcommand{\II}{\mathcal{I}}

\newcommand{\Amat}{\boldsymbol{A}}

\newcommand{\Imat}{\boldsymbol{I}}

\newcommand{\Xad}{X_\mathrm{ad}}

\def\di{\mathrm{d}}

\def\Dev{{\mathbf{Dev}\,}}

\def\enorm#1{|\hspace*{-.5mm}|\hspace*{-.5mm}|#1|\hspace*{-.5mm}|\hspace*{-.5mm}|}

\newcommand{\ip}[2]{(#1\hspace*{.5mm},#2)}
\newcommand{\dual}[2]{\langle#1\hspace*{.5mm},#2\rangle}

\newcommand{\norm}[3][]{#1\|#2#1\|_{#3}}

\def\curl{{\rm curl\,}}
\def\ccurl{{\bf curl\,}}

\def\div{{\rm div\,}}
\def\Div{{\mathbf{div}\,}}

\newcommand{\Hdivset}[1]{\boldsymbol{H}(\div;#1)}
\newcommand{\HDivset}[1]{\underbar\HH(\Div;#1)}
\newcommand{\Hcurlset}[1]{\boldsymbol{H}(\ccurl;#1)}
\newcommand{\HcurlsetZero}[1]{\boldsymbol{H}_0(\ccurl;#1)}

\newcommand{\set}[2]{\big\{#1\,:\,#2\big\}}

\newcommand{\pp}{\boldsymbol{p}}

\newcommand{\RT}{\ensuremath{\mathcal{RT}}}

\newcommand{\R}{\ensuremath{\mathbb{R}}}

\newcommand{\HH}{\ensuremath{{\boldsymbol{H}}}}

\newcommand{\LL}{\ensuremath{\mathcal{L}}}

\newcommand{\vv}{\ensuremath{\boldsymbol{v}}}
\newcommand{\ww}{\ensuremath{\boldsymbol{w}}}
\newcommand{\TT}{\ensuremath{\mathcal{T}}}
\newcommand{\FF}{\ensuremath{\boldsymbol{F}}}

\newcommand{\cS}{\ensuremath{\mathcal{S}}}

\newcommand{\PP}{\ensuremath{\mathcal{P}}}

\newcommand{\OO}{\ensuremath{\mathcal{O}}}

\newcommand{\NN}{\ensuremath{\boldsymbol{N}}}
\newcommand{\zz}{\ensuremath{{\boldsymbol{z}}}}
\newcommand{\normal}{\ensuremath{{\boldsymbol{n}}}}

\renewcommand{\AA}{\ensuremath{\mathcal{A}}}
\newcommand{\BB}{\ensuremath{\mathcal{B}}}

\newcommand{\bb}{\ensuremath{\boldsymbol{b}}}
\renewcommand{\aa}{\ensuremath{\boldsymbol{a}}}
\newcommand{\ff}{\ensuremath{\boldsymbol{f}}}
\newcommand{\CC}{\ensuremath{\mathcal{C}}}

\newcommand{\ran}{\operatorname{ran}}

\newcommand\DD{\mathcal{D}}

\newcommand{\ssigma}{{\boldsymbol\sigma}}
\newcommand{\xxi}{{\boldsymbol\xi}}
\newcommand{\ttau}{{\boldsymbol\tau}}

\newcommand{\qq}{{\boldsymbol{q}}}
\newcommand{\uu}{\boldsymbol{u}}

\newcommand{\MM}{\boldsymbol{M}}

\newcommand{\tr}{\operatorname{tr}}

\newcommand{\xx}{\boldsymbol{x}}

\newcommand{\yy}{\boldsymbol{y}}

\begin{document}

\title{Least-squares finite elements for distributed optimal control problems}
\date{\today}

\author{Thomas F\"uhrer}
\address{Facultad de Matem\'{a}ticas, Pontificia Universidad Cat\'{o}lica de Chile, Santiago, Chile}
\email{tofuhrer@mat.uc.cl}

\author{Michael Karkulik}
\address{Departamento de Matem\'atica, Universidad T\'ecnica Federico Santa Mar\'ia, Valpara\'iso, Chile}
\email{michael.karkulik@usm.cl}

\thanks{{\bf Acknowledgment.} 
This work was supported by ANID through FONDECYT projects and 1210391 (TF), and 1210579 (MK)}

\begin{abstract}
  We provide a framework for the numerical approximation of distributed optimal control problems, based on least-squares finite element methods. 
  Our proposed method simultaneously solves the state and adjoint equations and is $\inf$--$\sup$ stable for any choice of conforming discretization spaces. 
  A reliable and efficient a posteriori error estimator is derived for problems where box constraints are imposed on the control. It can be localized and therefore used to steer an adaptive algorithm.
  For unconstrained optimal control problems, i.e., the set of controls being a Hilbert space, we obtain a coercive least-squares method and,
    in particular, quasi-optimality for any choice of discrete approximation space. For constrained problems we derive and analyze a variational inequality where the PDE part is tackled by least-squares finite element methods. 
  We show that the abstract framework can be applied to a wide range of problems, including scalar second-order PDEs, the Stokes problem, and parabolic problems on space-time domains. 
  Numerical examples for some selected problems are presented. 
\end{abstract}
\maketitle

\section{Introduction}
Optimal control problems subject to PDEs form an important class of problems in practice, see, e.g.,~\cite{Lions71} for various applications.
Nowadays, the theory of distributed optimal control problems subject to linear PDEs is well understood, and
the most common approach to solve such problems is by deriving first-order optimality conditions and introducing either the adjoint state or
a Lagrangian multiplier. This approach leads, in general, to a symmetric but indefinite system.
Quite naturally, finite element techniques (FEM) have been used to discretize the resulting formulations, see~\cite[Ch.11]{BochevGunzberger09}.
Standard references on different finite element methods used in optimal control are~\cite{BoffiBrezziFortin,BrennerScott,ErnGuermond}.
In the present work, we particularly consider least-squares finite element methods, which are treated thoroughly
in~\cite{BochevGunzberger09}.
We also refer to~\cite[Ch.11]{BochevGunzberger09} for an overview on references and different techniques to treat optimal control
problems with FEM resp. least-squares FEM (LSFEM).

In the present work we consider a different approach, inspired by the work~\cite{BochevGunzberger06}.
In~\cite[Sec.4]{BochevGunzberger06} the optimality system is written in operator form and, then, a quadratic functional is defined by summing the squared norms of the residuals. 
Minimizing this functional is equivalent to solving the optimal control problem, and minimizing over finite element spaces gives a LSFEM.
The method from~\cite[Sec.4]{BochevGunzberger06} is defined when the control variable is sought in a Hilbert space.
However, when the control is restricted to a convex subset only --- a situation encountered in many practical applications --- there is no simple, i.e., linear,
relation between the adjoint state and the control.
Here, we consider the general case of optimal control problems where the set of admissible controls is a convex closed subset of the Lebesgue space of square-integrable functions.
Starting from the optimality system consisting of the state equation, the adjoint state equation and a variational inequality relating the adjoint state and the control, see~\cite[Ch.~2]{Lions71}, our proposed method is based on a least-squares functional (for the state and adjoint state equation) and a duality term corresponding to the variational inequality of the optimality system.
The whole method results in a variational inequality and we prove that the bilinear form, though not symmetric, is coercive, thus, the discrete systems are invertible for any choice of discretization space. 
The proposed method can also be interpreted as a Nitsche-type coupling of a LSFEM and a Galerkin method.
Nitsche-type methods are defined by adding terms to a bilinear form to include, e.g., essential boundary conditions in a weak sense, see, e.g.,~\cite{Stenberg95}.
In the past, we have used such a coupling technique for other minimum residual methods like the discontinuous Petrov--Galerkin method with optimal test functions (DPG) to define a coupling to boundary element methods~\cite{DPGBEM} or for the analysis of the DPG method for Signorini problems~\cite{DPGSignorini}.
If the set of admissible controls is a closed subspace, then our method simplifies to a LSFEM method that is essentially a reduced version of the method
from~\cite[Sec.4]{BochevGunzberger06} (we eliminated the optimality condition from~\cite[Eq.(4.1)]{BochevGunzberger06}).

In this article, we restrict the presentation to a framework where we consider LSFEMs which minimize first-order system residuals in $L^2$ norms only.
The practical implications of this framework are that standard finite element spaces (continuous and piecewise polynomial) are conforming, and that
the system matrices are easily computable up to numerical quadrature. A disadvantage is an increased number of degrees of freedom due
to the introduction of a new variable.
We show that the framework is quite general in the sense that it is applicable to a wide range of problems, including distributed optimal control problems
subject to reaction--diffusion--convection problems, the Stokes problem, the heat equation and many more.
The heat equation, in particular, will be formulated in space-time following methods developed in the recent works~\cite{FuehrerK_21,GS21}.
The advantage of this approach is that it is robust in standard energy norms even on locally refined space-time
meshes and that standard finite element spaces are conforming.
These space-time discretizations have been used recently for the optimal control of parabolic problems in~\cite{OptControlHeat,OptControlHeatGS22}.
While the latter two works use a Lagrangian multiplier in order to obtain a complete optimality system, in the present work we directly include
the adjoint state equation. An advantage of this approach is that we immediately obtain an approximation of the adjoint
state without any post-processing. We also stress that robustness of the employed space-time discretization allows us to use final-time desired states,
cf.~\cite{OptControlHeat}.
On the downside, using first-order formulations for the state and the adjoint-state equations implies an increased number of degrees of freedom. 
In general, this is also a critical point in space-time discretizations. However, in transient problems such as optimal control for parabolic equations, the entire history
of the state needs to be stored anyway, even if time-stepping methods are employed.
For unconstrained optimal control problems our proposed method yields a symmetric and positive definite system of equations, whereas the approaches from~\cite{OptControlHeat,OptControlHeatGS22} yield symmetric saddle-point systems. 
Furthermore, in the present work the control is eliminated (for unconstrained problems) which leads to slightly smaller dimensions of the discrete spaces. 
Regarding convergence we note that all three methods use the same Hilbert spaces defined over the space-time cylinder and are quasi-optimal with respect to canonical norms.
For constrained optimal control problems our proposed method requires us to choose sufficiently large parameters to ensure invertibility where the optimal choice of these parameters is not explicitly known. 
For unconstrained problems, however, our proposed method as well as the works~\cite{OptControlHeat,OptControlHeatGS22} do not require sufficiently large parameters.
Other works on space-time methods for the optimal control of parabolic problems include~\cite{GongHinzeZhou12,GunzburgerK_11,LangerSTY21,LangerSTY_21_a,MeidnerV_07}.
These works use different space-time discretization methods such as residual minimization for second order PDEs, wavelet methods,
mixed finite element methods, or discontinuous Galerkin methods, which do not share all the advantages as the space-time discretization employed in the present work.

Furthermore, we derive an a posteriori error estimator that is fully localizable as well as reliable and efficient in the case that box constraints are imposed on the control.
The estimator consists of the norms of the residuals of the state and adjoint state equations plus the error between the discrete computed control and the projection of a similar quantity defined by the discrete adjoint state.
In contrast to the estimator from~\cite{OptControlHeat} the a posteriori error estimator from the work at hand does not rely on solving an auxiliary problem.

The remainder of this article is organized as follows: In Section~\ref{sec:main} we introduce the underlying assumptions of the LSFEMs under consideration, derive the optimality system in the notation of LSFEMs and define a general LSFEM based method to solve distributed optimal control problems. 
A well-posedness analysis and quasi-best approximation estimates are given together with the derivation of reliable and efficient error estimators. 
Section~\ref{sec:examples} gives various examples where the abstract framework is applicable and Section~\ref{sec:num} presents numerical experiments for some selected problems.

\section{LSFEM based framework}\label{sec:main}
In this section we present a framework using least-squares finite elements for optimal control problems.
In subsection~\ref{sec:optcont} we describe the optimal control problem and assumptions tailored to the use of LSFEMs.
A LSFEM based method for the constrained optimal control problem is defined in subsection~\ref{sec:lsfem:constrained} and a priori as well as a posteriori analysis is provided. 
For unconstrained optimal control problems we show in subsection~\ref{sec:lsfem:unconstrained} that our proposed method can be simplified to a pure LSFEM problem.

\subsection{Notation}
For a Lipschitz domain $\Omega\subset \R^d$, $d\geq 1$ we use common notations for Lebesgue and Sobolev spaces, e.g., $L^2(\Omega;\R)$, $L^2(\Omega;\R^d)$, $L^2(\Omega;\R^{d\times d})$, $H^1(\Omega)$, $H_0^1(\Omega)$. 
We use $\ip\cdot\cdot_{L^2(\Omega)}$ to denote the canonical inner product in $L^2(\Omega;\R)$, $L^2(\Omega;\R^d)$ or $L^2(\Omega;\R^{d\times d})$. The induced norm is denoted by $\norm\cdot{L^2(\Omega)}$.
Throughout, we identify Lebesgue-like spaces with their duals. 
The Sobolev space $H_0^1(\Omega)$ is equipped with the norm $\norm{\nabla(\cdot)}{L^2(\Omega)}$. We also need 
\begin{align*}
  \Hdivset\Omega = \set{\ssigma\in L^2(\Omega;\R^d)}{\div\ssigma\in L^2(\Omega;\R)}
\end{align*}
with norm $\norm{\ssigma}{\Hdivset\Omega}^2 = \norm{\ssigma}{L^2(\Omega)}^2 + \norm{\div\ssigma}{L^2(\Omega)}^2$. 
For a generic Hilbert space $H$ we denote by $\ip{\cdot}{\cdot}_H$ its inner product and by $\norm\cdot{H}$ the induced norm.

If $X_1,\dots,X_n$ are Hilbert spaces, then we equip the product space $X=X_1\times\dots\times X_n$ with the norm
\begin{align*}
  \norm{(x_1,\dots,x_n)}{X}^2 = \norm{x_1}{X_1}^2 + \dots + \norm{x_n}{X_n}^2. 
\end{align*}

For the functional analytic setting of the heat equation we use Bochner spaces. Let $X$ be a Hilbert space and $I=(0,T)$ a time interval. 
The elements of the Bochner space $L^2(I;X)$ are functions $x\colon I\to X$ which are strongly measurable with respect to the Lebesgue measure $\di t$ with
\begin{align*}
  \int_I \norm{x(t)}X^2 \,\di t < \infty. 
\end{align*}
Similarly, the Bochner space $H^1(I;X)$ is defined as the space of functions $x\in L^2(X)$ such that $\partial_t x \in L^2(X)$, where $\partial_t(\cdot)$ denotes the weak derivative with respect to time. 
We refer to, e.g.,~\cite[Sec.5.9.2]{Evans98}, for a short introduction to Bochner spaces.

\subsection{Optimal control problem}\label{sec:optcont}
In this work we consider the following optimal control problem. 
Let $\ff\in Z$ be given and let the \emph{desired state} $\zz_d\in H$ be given. 
With the \emph{cost parameter} $0<\lambda\leq 1$ define the \emph{cost functional} $J\colon X\to\R$,
\begin{align}\label{eq:optCtr}
  \uu\mapsto \norm{\AA\II\yy(\uu) -\zz_d}{H}^2 + \lambda\ip{\CC\uu}{\uu}_X,
\end{align}
where the \emph{state} $\yy=\yy(\uu)\in Y$ is the unique solution of the \emph{state equation}
\begin{align*}
  \LL\yy + \BB\uu = \ff .
\end{align*}

The next subsection collects all the assumptions on the spaces and operators for problem~\eqref{eq:optCtr}.
\subsubsection{Assumptions}\label{sec:ass}
Let $\Omega_1,\dots,\Omega_n$ denote Lipschitz domains.
We consider the following spaces resp. sets,
\begin{itemize}
  \item $H$, $X$, $Y$, and $Y^\star$ Hilbert spaces,
  \item $\Xad\subseteq X$ a non-empty, closed and convex subset, 
  \item $Z= \boldsymbol{L}^2(\Omega_1)\times \cdots \times \boldsymbol{L}^2(\Omega_n)$.
\end{itemize}
Here, $\boldsymbol{L}^2(\Omega_j)$ stands for a Lebesgue space, e.g., $\boldsymbol{L}^2(\Omega_j)= L^2(\Omega_j,\R)$ or
$\boldsymbol{L}^2(\Omega_j)=L^2(\Omega_j;\R^d)$ (with $d$ depending on $\Omega_j$).
We also consider bounded linear operators, 
\begin{itemize}
  \item $\AA\colon Z\to H$, with adjoint $\AA^*\colon H\to Z$, 
  \item $\BB\colon X\to Z$, with adjoint $\BB^*\colon Z\to X$, 
  \item $\CC\colon X\to X$ self-adjoint and positive definite, i.e., 
    \begin{align}\tag{A1}\label{ass:C}
      \CC=\CC^* \text{ and there exists } \kappa>0 \text{ such that }
      \ip{\CC\uu}{\uu}_X \geq \kappa \norm{\uu}{X}^2 \quad\forall \uu\in X.
    \end{align}
\end{itemize}
Moreover, we consider linear operators,
\begin{align*}
  \LL\colon Y\to Z, \quad
  \LL^\star \colon Y^\star\to Z,
\end{align*}
and assume that
\begin{align}\tag{A2}\label{ass:L}
  \LL\text{ and }\LL^\star \text{ are bounded}, \, \ker(\LL)=\{0\}=\ker(\LL^\star), \quad \ran(\LL)=Z=\ran(\LL^\star),
\end{align}
i.e., $\LL$ and $\LL^\star$ are boundedly invertible. Furthermore, we assume that there exist bounded linear operators 
\begin{align*}
  \II&\colon Y\to Z \quad \text{and}\quad \II^\star \colon Y^\star\to Z
\end{align*}
such that $\LL$ and $\LL^\star$ are adjoint in the sense that
\begin{align}\tag{A3}\label{ass:I}
  \ip{\LL\yy}{\II^\star\pp}_Z = \ip{\II\yy}{\LL^\star\pp}_Z \quad\forall \yy\in Y, \pp\in Y^\star.
\end{align}
We anticipate that $\LL$ resp. $\LL^\star$ will correspond to least-squares formulations of
PDEs. Various examples will be given in Section~\ref{sec:examples}.

\begin{remark}
  As stated in the introduction we consider LSFEMs which minimize residuals in $L^2$ norms. 
  It is possible to replace $Z$ by a general Hilbert space.
  To solve the operator equation $\LL\yy = \ff$ in $Z$, a least-squares formulation would read: Find the minimizer of
  \begin{align*}
    \min_{\yy_h\in Y_h} \norm{\LL\yy_h-\ff}{Z}^2.
  \end{align*}
  This requires to implement the inner product in $Z$ which may require the inversion of certain operators, e.g., if $Z$ is the dual space of $H_0^1(\Omega)$.
  For some problems it is possible to replace the inner product by a discrete version, see, e.g.,~\cite{BLP97}.
\end{remark}%

\begin{remark}
  In most applications $\II$ and $\II^\star$ are canonical embedding operators, e.g., $H_0^1(\Omega)\times \Hdivset\Omega \hookrightarrow L^2(\Omega)\times L^2(\Omega;\R^d)$.
  For the formulation of the heat equation considered below (Section~\ref{sec:examples:heat}) these operators also include a trace operation (restriction to initial resp. end time).
\end{remark}%

\subsubsection{Auxiliary results}
By following the proof of~\cite[Ch.~2, Theorem~1.4]{Lions71} we obtain the next result.
\begin{proposition}
  The minimization problem to find $\uu\in \Xad$ such that
  \begin{align*}
    J(\uu) = \min_{\vv\in \Xad} J(\vv)
  \end{align*}
  is equivalent to
  \begin{subequations}\label{eq:model}
  \begin{align}
    \LL\yy &= \ff-\BB\uu,\label{eq:model:a} \\
    \LL^\star\pp &= \AA^*(\AA\II\yy-\zz_d),\label{eq:model:b} \\
    \ip{-\BB^*\II^\star\pp+\lambda\CC\uu}{\vv-\uu}_X &\geq 0 \label{eq:model:c}
  \end{align}
  for all $\vv\in \Xad$.
  \end{subequations}
\end{proposition}
\begin{proof}
  Note that the functional $J$ is differentiable and convex. The minimization problem therefore is equivalent, see, e.g.~\cite[Lemma~2.21]{Troeltzsch}, to the variational inequality
  \begin{align*}
    \ip{\AA\II\yy(\uu)-\zz_d}{\AA\II(\yy(\vv)-\yy(\uu))}_H + \ip{\lambda\CC\uu}{\vv-\uu}_X \geq 0
    \quad\forall \vv\in \Xad.
  \end{align*}
  Define $\pp\in Y^\star$ as the unique solution of 
  \begin{align*}
    \LL^\star\pp = \AA^*(\AA\II\yy(\uu)-\zz_d)
  \end{align*}
  which is possible due to~\eqref{ass:L}.
  Then, using~\eqref{ass:I} gives
  \begin{align*}
    \ip{\AA\II\yy(\uu)-\zz_d}{\AA\II(\yy(\vv)-\yy(\uu))}_H &= \ip{\AA^*(\AA\II\yy(\uu)-\zz_d)}{\II(\yy(\vv)-\yy(\uu))}_Z
    \\
    &= \ip{\LL^\star \pp}{\II(\yy(\vv)-\yy(\uu))}_Z = \ip{\II^\star\pp}{\LL(\yy(\vv)-\yy(\uu))}_Z.
  \end{align*}
  Employing that $\LL\yy(\uu) = \ff-\BB\uu$ and $\LL\yy(\vv)=\ff-\BB\vv$ we end up with
  \begin{align*}
    \ip{\AA\II\yy(\uu)-\zz_d}{\AA\II(\yy(\vv)-\yy(\uu))}_H = -\ip{\II^\star\pp}{\BB(\vv-\uu)}_Z = \ip{-\BB^*\II^\star\pp}{\vv-\uu}_X.
  \end{align*}
  This leads to the variational inequality
  \begin{align*}
    \ip{-\BB^*\II^\star\pp+\lambda\CC\uu}{\vv-\uu}_X\geq 0 \quad\forall \vv \in \Xad
  \end{align*}
  which finishes the proof.
\end{proof}

For the analysis below we make use of a parameter-dependent norm on the product space $X\times Y\times Y^\star$ given for all $(\uu,\yy,\pp)\in X\times Y\times Y^\star$ by
\begin{align*}
  \enorm{(\uu,\yy,\pp)}_\lambda^2 &:= \lambda\norm{\uu}X^2 + \norm{\yy-\yy_{\uu}}{Y}^2 + \lambda^{-1}\norm{\pp-\pp_{\yy}}{Y^\star}^2 + \norm{\AA\II\yy_{\uu}}H^2, 
\end{align*}
where for each $\uu\in X$, $\yy\in Y$ the functions $\yy_{\uu}\in Y$ and $\pp_{\yy}\in Y^\star$ are defined as solutions of
\begin{align*}
  \LL\yy_{\uu} = -\BB\uu, \quad \LL^\star\pp_{\yy} = \AA^*\AA\II\yy.
\end{align*}
From this point on we write $A\lesssim B$ for $A,B\geq 0$ if there exists a generic constant $C>0$ independent of the cost parameter $\lambda$ such that $A\leq C\cdot B$. 
If both inequalities hold, i.e., $A\lesssim B \lesssim A$ then we write $A\eqsim B$.
\begin{lemma}\label{lem:normequiv}
  The mapping $\enorm{\cdot}_\lambda\colon X\times Y\times Y^\star\to \R_{\geq 0}$, $(\uu,\yy,\pp)\mapsto \enorm{(\uu,\yy,\pp)}_\lambda$ defines a norm on the product space $X\times Y\times Y^\star$ with
  \begin{align*}
    \lambda^{1/2} \norm{(\uu,\yy,\pp)}{X\times Y\times Y^\star} \lesssim \enorm{(\uu,\yy,\pp)}_\lambda \lesssim \lambda^{-1/2} \norm{(\uu,\yy,\pp)}{X\times Y\times Y^\star} \quad\forall (\uu,\yy,\pp)\in X\times Y\times Y^\star.
  \end{align*}
\end{lemma}
\begin{proof}
  By linearity of the definitions $\yy_{\uu}$, $\pp_{\yy}$ the triangle inequality and homogeneity of $\enorm{\cdot}_\lambda$ follow.
  Definiteness can be seen from the equivalence estimate. Let $(\uu,\yy,\pp)\in X\times Y\times Y^\star$ be given. By the definition of $\yy_{\uu}$ and $\pp_{\yy}$ and the assumptions of Section~\ref{sec:ass} we have 
  $\norm{\yy_{\uu}}Y\lesssim \norm{\uu}X$ and $\norm{\pp_{\yy}}{Y^\star}\lesssim \norm{\yy}{Y}$.
  Using the latter estimates together with the triangle inequality proves
  \begin{align*}
    \lambda \norm{(\uu,\yy,\pp)}{X\times Y\times Y^\star}^2 &= \lambda\norm{\uu}{X}^2 + \lambda\norm{\yy}{Y}^2 + \lambda \norm{\pp}{Y^\star}^2 
    \\
    &\lesssim \lambda\norm{\uu}{X}^2 + \lambda\norm{\yy}{Y}^2 + \lambda \norm{\pp-\pp_{\yy}}{Y^\star}^2\\
    &\lesssim \lambda\norm{\uu}{X}^2 + \lambda\norm{\yy-\yy_{\uu}}{Y}^2 + \lambda \norm{\pp-\pp_{\yy}}{Y^\star}^2 \\
    &\lesssim \lambda\norm{\uu}{X}^2 + \norm{\yy-\yy_{\uu}}{Y}^2 + \lambda^{-1} \norm{\pp-\pp_{\yy}}{Y^\star}^2 + \norm{\AA\II\yy_{\uu}}H^2
    \\
    &= \enorm{(\uu,\yy,\pp)}_\lambda^2,
  \end{align*}
  where we used $0 < \lambda\leq 1$.
  For the upper bound the very same arguments show that
  \begin{align*}
    \enorm{(\uu,\yy,\pp)}_\lambda^2 &= \lambda\norm{\uu}{X}^2 + \norm{\yy-\yy_{\uu}}{Y}^2 + \lambda^{-1} \norm{\pp-\pp_{\yy}}{Y^\star}^2 + \norm{\AA\II\yy_{\uu}}H^2
    \\
    &\lesssim \lambda\norm{\uu}{X}^2 + \lambda^{-1}\norm{\yy}{Y}^2 + \lambda^{-1} \norm{\pp}{Y^\star}^2 
    + \norm{\yy_{\uu}}{Y}^2 + \lambda^{-1} \norm{\pp_{\yy}}{Y^\star}^2 + \norm{\uu}{X}^2
    \\
    &\lesssim \norm{\uu}{X}^2 + \lambda^{-1}\norm{\yy}{Y}^2 + \lambda^{-1} \norm{\pp}{Y^\star}^2 \leq \lambda^{-1}\norm{(\uu,\yy,\pp)}{X\times Y\times Y^\star}^2. 
  \end{align*}
  This finishes the proof.
\end{proof}

\subsection{LSFEM for constrained optimal control problem}\label{sec:lsfem:constrained}
Let us define the bilinear form $a\colon (X\times Y\times Y^\star)^2\to \R$ by
\begin{align*}
  a(\uu,\yy,\pp;\vv,\zz,\qq) &= \alpha\ip{\LL\yy+\BB\uu}{\LL\zz+\BB\vv}_Z + \beta\ip{\LL^\star\pp-\AA^*\AA\II\yy}{\LL^\star\qq-\AA^*\AA\II\zz}_Z
  \\
  &\qquad + \ip{-\BB^*\II^\star\pp+\lambda\CC\uu}{\vv}_X
\end{align*}
for $\uu,\vv\in X$, $\yy,\zz\in Y$, $\pp,\qq\in Y^\star$. Here $\alpha,\beta$ denote positive constants.
We also consider the load functional $\ell\colon X\times Y\times Y^\star\to \R$, for given $\ff\in Z,\zz_d\in H$, defined by
\begin{align*}
  \ell(\vv,\zz,\qq) = \alpha\ip{\ff}{\LL\zz+\BB\vv}_Z + \beta\ip{-\AA^*\zz_d}{\LL^\star\qq-\AA^*\AA\II\zz}_Z.
\end{align*}

We now show that~\eqref{eq:model} is equivalent to the following variational inequality: Find $(\uu,\yy,\pp)\in \Xad\times Y\times Y^\star$ such that
\begin{align}\label{eq:varin:constrained}
  a(\uu,\yy,\pp;\vv-\uu,\zz-\yy,\qq-\pp) \geq \ell(\vv-\uu,\zz-\yy,\qq-\pp) \quad\text{for all } 
  (\vv,\zz,\qq)\in \Xad \times Y \times Y^\star.
\end{align}

\begin{proposition}
  Problems~\eqref{eq:model} and~\eqref{eq:varin:constrained} are equivalent. 
\end{proposition}
\begin{proof}
  If $\uu,\yy,\pp$ solve~\eqref{eq:model}, then they also satisfy~\eqref{eq:varin:constrained}. 

  Let $\uu,\yy,\pp$ be a solution of~\eqref{eq:varin:constrained}. 
  By testing~\eqref{eq:varin:constrained} with $\vv=\uu$, $\zz = \yy$, $\qq=\pp\pm\ww$, $\ww\in Y^\star$, we see that
  \begin{align*}
    \pm\beta\ip{\LL^\star\pp-\AA^*\AA\II\yy}{\LL^\star\ww}_Z\geq \pm \beta\ip{-\AA^*\zz_d}{\LL^\star\ww}_Z, 
  \end{align*}
  or equivalently
  \begin{align*}
    \ip{\LL^\star\pp-\AA^*\AA\II\yy}{\LL^\star\ww}_Z = \ip{-\AA^*\zz_d}{\LL^\star\ww}_Z
  \end{align*}
  for all $\ww\in Y^\star$. With~\eqref{ass:L}, this yields
  \begin{align*}
    \LL^\star\pp = \AA^*(\AA\II\yy-\zz_d),
  \end{align*}
  which is~\eqref{eq:model:b}.
  By testing~\eqref{eq:varin:constrained} with $\vv=\uu$, $\zz=\yy\pm\ww$, $\ww\in Y$, $\qq=\pp$, and using the latter identity
  we further conclude with a similar argumentation as before that
  \begin{align*}
    \ip{\LL\yy+\BB\uu}{\LL\ww}_Z = \ip{\ff}{\LL\ww}_Z \quad\forall \ww\in Y,
  \end{align*}
  which means that, using again~\eqref{ass:L},
  \begin{align*}
    \LL\yy + \BB\uu = \ff, 
  \end{align*}
  i.e.,~\eqref{eq:model:a}.  Using the identities~\eqref{eq:model:a}--\eqref{eq:model:b}, we see that
  the variational inequality~\eqref{eq:varin:constrained} reduces to
  \begin{align*}
    0&\leq a(\uu,\yy,\pp;\vv-\uu,\zz-\yy,\qq-\pp)-\ell(\vv-\uu,\zz-\yy,\qq-\pp) 
    \\&= \ip{-\BB^*\II^\star\pp+\lambda\CC\uu}{\vv-\uu}_X 
  \end{align*}
  for all $\vv\in \Xad$. This is~\eqref{eq:model:c} and, thus, finishes the proof.
\end{proof}

The next result shows that the bilinear form $a$ is coercive and bounded, given that the constants $\alpha,\beta$ are sufficiently large.
In other words, $a$ induces an equivalent norm on the underlying energy space, and thus Problem~\eqref{eq:varin:constrained} admits a unique solution.

\begin{theorem}\label{thm:constrained}
  There exists $\alpha_0,\beta_0$ with $\alpha_0\eqsim 1$, $\beta_0\eqsim \lambda^{-1}$ such that for all $\alpha\geq\alpha_0$, $\beta\geq \beta_0$ and all $(\uu,\yy,\pp),(\vv,\zz,\qq)\in X\times Y \times Y^\star$, %
  \begin{align*}
    \enorm{(\uu,\yy,\pp)}_\lambda^2 &\lesssim a(\uu,\yy,\pp;\uu,\yy,\pp),    \\
    |a(\uu,\yy,\pp;\vv,\zz,\qq)| &\lesssim \max\{\alpha,\lambda\beta\}\enorm{(\uu,\yy,\pp)}_\lambda \enorm{(\vv,\zz,\qq)}_\lambda,
  \end{align*}%
  i.e., the bilinear form $a(\cdot,\cdot)$ is coercive and bounded with respect to the norm $\enorm{\cdot}_\lambda$.
  For any non-empty closed convex subset $W\subseteq X\times Y\times Y^\star$, the variational inequality to find $(\uu,\yy,\pp)\in W$ such that
  \begin{align}\label{eq:varineq:general}
    a(\uu,\yy,\pp;\vv-\uu,\zz-\yy,\qq-\pp) \geq \ell(\vv-\uu,\zz-\yy,\qq-\pp) \quad\text{for all } 
    (\vv,\zz,\qq)\in W
  \end{align}
  admits a unique solution. In particular, Problem~\eqref{eq:model} respectively~\eqref{eq:varin:constrained} admits a unique solution.
\end{theorem}
\begin{proof}
  It suffices to prove that the bilinear form is coercive and bounded and that the right-hand side is linear and bounded.
  Existence and uniqueness of the given variational inequality then follow from the Lions--Stampacchia theorem, cf.~\cite{LionsS_67}.
  Let $(\uu,\yy,\pp),(\vv,\zz,\qq)\in X\times Y \times Y^\star$ be given. Recall that $\yy_{\uu}$ and $\pp_{\yy}$ are given as solutions of
  \begin{align*}
    \LL\yy_{\uu} = -\BB\uu, \quad \LL^\star \pp_{\yy} = \AA^*\AA\II\yy.
  \end{align*}
By the assumptions of Section~\ref{sec:ass} we have $\norm{\yy_{\uu}}Y\lesssim \norm{\uu}{X}$ and $\norm{\pp_{\yy}}{Y^\star}\lesssim \norm{\yy}{Y}$.
Moreover, 
  \begin{align*}
    \norm{\yy-\yy_{\uu}}{Y} \eqsim \norm{\LL(\yy-\yy_{\uu})}{Z} = \norm{\LL\yy+\BB\uu}{Z}
  \end{align*}
  and
  \begin{align*}
    \norm{\pp-\pp_{\yy}}{Y^\star}&\eqsim \norm{\LL^\star(\pp-\pp_{\yy})}Z = \norm{\LL^\star\pp-\AA^*\AA\II\yy}{Z}.
  \end{align*}
  Putting the last two estimates together proves
  \begin{align}\label{eq:proofcoerc1}
    \norm{\LL\yy+\BB\uu}{Z} + \lambda^{-1/2}\norm{\LL^\star\pp-\AA^*\AA\II\yy}{Z} \eqsim  \norm{\yy-\yy_{\uu}}{Y} + \lambda^{-1/2}\norm{\pp-\pp_{\uu}}{Y^\star} \lesssim \enorm{(\uu,\yy,\pp)}_\lambda,
  \end{align}
  which is used in the remainder of the proof. 

  To prove boundedness we consider the individual terms in the bilinear form. First, by the Cauchy--Schwarz inequality and~\eqref{eq:proofcoerc1} we see that 
  \begin{align*}
    \alpha|\ip{\LL\yy+\BB\uu}{\LL\zz+\BB\vv}_Z| &\lesssim \alpha \norm{\LL\yy+\BB\uu}Z\norm{\LL\zz+\BB\vv}Z \\
    &\lesssim \alpha \enorm{(\uu,\yy,\pp)}_\lambda\enorm{(\vv,\zz,\qq)}_\lambda.
  \end{align*}
  Second, the same arguments yield
  \begin{align*}
    \beta |\ip{\LL^\star\pp-\AA^*\AA\II\yy}{\LL^\star\qq-\AA^*\AA\II\zz}_Z| &\leq \lambda\beta \lambda^{-1/2}\norm{\LL^\star\pp-\AA^*\AA\II\yy}Z \lambda^{-1/2}\norm{\LL^\star\qq-\AA^*\AA\II\zz}Z
    \\ 
    &\lesssim \lambda\beta \enorm{(\uu,\yy,\pp)}_\lambda\enorm{(\vv,\zz,\qq)}_\lambda.
  \end{align*}
  For the estimate of the final term in the bilinear form we note that 
  \begin{align}\label{eq:proofcoerc2}
  \begin{split}
    \ip{\BB^*\II^\star\pp_{\yy}}{\vv}_X &= \ip{\II^\star\pp_{\yy}}{\BB\vv}_Z = \ip{\II^\star\pp_{\yy}}{-\LL\yy_{\vv}}_Z
    \\
    &=-\ip{\LL^\star\pp_{\yy}}{\II\yy_{\vv}}_Z = -\ip{\AA^*\AA\II\yy}{\II\yy_{\vv}}_Z = -\ip{\AA\II\yy}{\AA\II\yy_{\vv}}_H.
  \end{split}
  \end{align}
  This yields 
  \begin{align*}
    |\ip{\BB^*\II^\star\pp_{\yy}}{\vv}_X| &\leq |\ip{\AA\II(\yy-\yy_{\uu})}{\AA\II\yy_{\vv}}_H| + |\ip{\AA\II\yy_{\uu}}{\AA\II\yy_{\vv}}_H|
    \\&\lesssim (\norm{\yy-\yy_{\uu}}Y+\norm{\AA\II\yy_{\uu}}H)\norm{\AA\II\yy_{\vv}}H \lesssim \enorm{(\uu,\yy,\pp)}_\lambda\enorm{(\vv,\zz,\qq)}_\lambda. 
  \end{align*}
  Then, 
  \begin{align*}
    |\ip{-\BB^*\II^\star\pp+\lambda\CC\uu}{\vv}_X| &\leq |\ip{\BB^*\II^\star(\pp-\pp_{\yy})}{\vv}_X| + |\ip{\BB^*\II^\star\pp_{\yy}}{\vv}_X|
    + \lambda |\ip{\CC\uu}{\vv}_X| 
    \\
    &\lesssim \lambda^{-1/2}\norm{\pp-\pp_{\yy}}{Y^\star}\lambda^{1/2}\norm{\vv}X + \enorm{(\uu,\yy,\pp)}_\lambda\enorm{(\vv,\zz,\qq)}_\lambda 
    \\ &\qquad + \lambda^{1/2}\norm{\uu}X\lambda^{1/2}\norm{\vv}X
    \\
    &\lesssim \enorm{(\uu,\yy,\pp)}_\lambda\enorm{(\vv,\zz,\qq)}_\lambda.
  \end{align*}
  This shows boundedness of the bilinear form. The very same arguments can be used to prove that the right-hand side functional $\ell$ is bounded. The details are omitted.

  It remains to prove coercivity.
  By the definition of the norm $\enorm\cdot_\lambda$ and~\eqref{eq:proofcoerc1} we get
  \begin{align*}
    \enorm{(\uu,\yy,\pp)}_\lambda^2 &= \norm{\yy-\yy_{\uu}}{Z}^2 + \lambda^{-1}\norm{\pp-\pp_{\yy}}Z^2 + \lambda \norm{\uu}X^2 + \norm{\AA\II\yy_{\uu}}H^2 
    \\
    &\eqsim \norm{\LL\yy+\BB\uu}Z^2 + \lambda^{-1}\norm{\LL^\star\pp-\AA^*\AA\II\yy}Z^2 
    + \lambda \ip{\CC\uu}{\uu}_X + \norm{\AA\II\yy_{\uu}}H^2.
    \\
    &= \norm{\LL\yy+\BB\uu}Z^2 + \lambda^{-1}\norm{\LL^\star\pp-\AA^*\AA\II\yy}Z^2 
    + \ip{-\BB^*\II^\star\pp+\lambda\CC\uu}{\uu}_X 
    \\
    &\qquad\qquad+ \ip{\BB^*\II^\star(\pp-\pp_{\yy})}{\uu}_X + \ip{\BB^*\II^\star\pp_{\yy}}{\uu}_X + \norm{\AA\II\yy_{\uu}}H^2.
  \end{align*}
  Using identity~\eqref{eq:proofcoerc2} with $\vv=\uu$ we see that 
  \begin{align*}
    \ip{\BB^*\II^\star\pp_{\yy}}{\uu}_X + \norm{\AA\II\yy_{\uu}}H^2 &= \ip{\AA\II(\yy_{\uu}-\yy)}{\AA\II\yy_{\uu}}_H
    \\
    &\leq C_1 \delta^{-1}/2 \norm{\LL\yy+\BB\uu}{Z}^2 + \delta/2 \norm{\AA\II\yy_{\uu}}H^2.
  \end{align*}
  In the last estimate we have used Young's inequality with $\delta>0$ and a constant $C_1>0$. 
  Employing Young's inequality again we further obtain
  \begin{align*}
    |\ip{\BB^*\II^\star(\pp-\pp_{\yy})}{\uu}_X| \leq  C_2 \delta^{-1}/2\lambda^{-1}\norm{\LL^\star\pp-\AA^*\AA\II\yy}{Z}^2 + \delta/2 \lambda\norm{\uu}{X}^2. 
  \end{align*}
  where $C_2>0$ denotes another positive constant.
  Combining the last two estimates together and using $\lambda\norm{\uu}{X}^2 + \norm{\AA\II\yy_{\uu}}H^2\leq \enorm{(\uu,\yy,\pp)}_\lambda^2$ we conclude that
  \begin{align*}
    \enorm{(\uu,\yy,\pp)}_\lambda^2 &\lesssim \norm{\LL\yy+\BB\uu}Z^2 + \lambda^{-1}\norm{\LL^\star\pp-\AA^*\AA\II\yy}Z^2 
    + \ip{-\BB^*\II^\star\pp+\lambda\CC\uu}{\uu}_X 
    \\
    &\qquad\qquad+ \ip{\BB^*\II^\star(\pp-\pp_{\yy})}{\uu}_X + \ip{\BB^*\II^\star\pp_{\yy}}{\uu}_X + \norm{\AA\II\yy_{\uu}}H^2
    \\
    &\leq (1+ C_1 \delta^{-1}/2) \norm{\LL\yy+\BB\uu}{Z}^2 + (1+C_2\delta^{-1}/2)\lambda^{-1}\norm{\LL^\star\pp-\AA^*\AA\II\yy}Z^2
    \\
    &\qquad\qquad+ \delta/2 \enorm{(\uu,\yy,\pp)}_\lambda^2
  \end{align*}
  Subtracting the last term on the right-hand side for sufficiently small $\delta>0$ and setting $\alpha_0 := (1+C_1\delta^{-1}/2)$, $\beta_0 :=\lambda^{-1}(1+C_2\delta^{-1})$ finishes the proof.
\end{proof}

\subsubsection{A priori error analysis}
In this section we derive quasi-optimality results using well-known techniques for variational inequalities, see, e.g.~\cite{Falk74}. 
In particular, following the proof of~\cite[Theorem~1]{Falk74} shows the next result.
\begin{theorem}\label{thm:apriori}
  Suppose that $1\eqsim\alpha\geq\alpha_0$ and $\lambda^{-1}\eqsim\beta\geq\beta_0$, such that the results of Theorem~\ref{thm:constrained} hold true. 
  Let $(\uu,\yy,\pp) \in \Xad\times Y\times Y^\star$ denote the unique solution of~\eqref{eq:varin:constrained}. 
  Let $X_{h,\mathrm{ad}}\subset \Xad$ denote a non-empty closed convex subset and $Y_h\subset Y$, $Y_h^\star\subset Y^\star$ closed subspaces.
  Set $W_h = X_{h,\mathrm{ad}}\times Y_h\times Y_h^\star$.
  If $(u_h,\yy_h,\pp_h)\in W_h$ denotes the solution of~\eqref{eq:varineq:general}, then
  \begin{align*}
    &\enorm{(\uu-\uu_h,\yy-\yy_h,\pp-\pp_h)}_\lambda^2
    \\ &\qquad\lesssim 
    \inf_{(\vv_h,\zz_h,\qq_h)\in W_h}\left( \enorm{(\uu-\vv_h,\yy-\zz_h,\pp-\qq_h)}_\lambda^2 + |\ip{-\BB^*\II^\star\pp+\lambda\CC \uu}{\vv_h-\uu}_{X}| \right).
  \end{align*}
\end{theorem}
\begin{proof}
  For the convenience of the reader we include the proof. 
  First, from~\eqref{eq:model:a}--\eqref{eq:model:b} we see that the solution $(\uu,\yy,\pp)$ satisfies
  \begin{align}\label{eq:apriori:proof1}
    a(\uu,\yy,\pp;\vv,\ww,\qq) = \ell(\vv,\ww,\qq)+\ip{-\BB^*\II^\star\pp+\lambda\CC\uu}{\vv}_X \quad\forall (\vv,\ww,\qq)\in X\times Y\times Y^\star.
  \end{align}
  We use the short notation $\xx = (\uu,\yy,\pp)$, $\xx_h = (\uu_h,\yy_h,\pp_h)$ and $\widetilde\xx_h=(\vv_h,\zz_h,\qq_h)\in W_h$. 
  Second, Theorem~\ref{thm:constrained} shows that the bilinear form is coercive.
  Together with~\eqref{eq:apriori:proof1}, this leads to
  \begin{align*}
    \enorm{\xx-\xx_h}_\lambda^2
    &\lesssim a(\xx-\xx_h;\xx-\xx_h)
    \\
    &= a(\xx;\xx-\xx_h) -a(\xx_h;\xx-\widetilde\xx_h)-a(\xx_h;\widetilde\xx_h-\xx_h) \\
    &\leq \ell(\xx-\xx_h)-a(\xx_h;\xx-\widetilde\xx_h)-\ell(\widetilde\xx_h-\xx_h)
    \\
    &= \ell(\xx-\widetilde\xx_h) - a(\xx_h;\xx-\widetilde\xx_h).
  \end{align*}
  Using~\eqref{eq:apriori:proof1}, boundedness of the bilinear form and Young's inequality with parameter $\delta>0$ we see that
  \begin{align*}
    \ell(\xx-\widetilde\xx_h) - a(\xx_h;\xx-\widetilde\xx_h) &= a(\xx;\xx-\widetilde\xx_h) - \ip{-\BB^*\II^\star\pp+\lambda\CC\uu}{\uu-\vv_h}_X - a(\xx_h;\xx-\widetilde\xx_h)
    \\
    &= a(\xx-\xx_h;\xx-\widetilde\xx_h) + \ip{-\BB^*\II^\star\pp+\lambda\CC\uu}{\vv_h-\uu}_X 
    \\
    &\lesssim \delta^{-1}\enorm{\xx-\widetilde\xx_h}_\lambda^2 
    + \ip{-\BB^*\II^\star\pp+\lambda\CC\uu}{\vv_h-\uu}_X
    + \delta\enorm{\xx-\xx_h}_\lambda^2.
  \end{align*}
  Putting all the estimates together and subtracting the last term for $\delta$ sufficiently small proves the asserted quasi-optimality result.
\end{proof}

An immediate consequence of Theorem~\ref{thm:apriori} and norm equivalence from Lemma~\ref{lem:normequiv} is the following result.
\begin{corollary}\label{cor:apriori}
  With the assumptions and notations of Theorem~\ref{thm:apriori}, the estimate
  \begin{align*}
    &\lambda\norm{(\uu-\uu_h,\yy-\yy_h,\pp-\pp_h)}{X\times Y\times Y^\star}^2
    \\ &\qquad\lesssim 
    \inf_{(\vv_h,\zz_h,\qq_h)\in W}\left( \lambda^{-1}\norm{(\uu-\vv_h,\yy-\zz_h,\pp-\qq_h)}{X\times Y\times Y^\star}^2 + |{\ip{-\BB^*\II^\star\pp+\lambda\CC \uu}{\vv_h-\uu}_{X}}| \right)
  \end{align*}
  holds. 
\end{corollary}%

\begin{remark}
  If $X_{h,\mathrm{ad}}$ is not a subset of $\Xad$, i.e., if we consider a non-conforming discretization of the convex set $\Xad$, then the additional term
  \begin{align*}
    \inf_{\vv\in \Xad} |\ip{-\BB^*\II^\star \pp_h + \lambda\CC \uu_h}{\vv-\uu_h}_X|.
  \end{align*}
  enters on the right-hand side of the a priori estimate in Theorem~\ref{thm:apriori}.
\end{remark}

\subsubsection{A posteriori error analysis}
In this section we derive reliable and efficient a posteriori error estimators.
Let us recall the variational inequality from~\eqref{eq:model}, 
\begin{align*}
  \ip{-\BB^*\II^\star \pp + \lambda\CC\uu}{\vv-\uu}_X\geq 0 \quad\forall \vv\in \Xad.
\end{align*}
By our assumptions on the operator $\CC$ we have that $\ip{\cdot}\cdot_{\CC}:=\ip{\CC(\cdot)}{(\cdot)}_X$ defines an inner product on $X$ which induces the norm $\norm\cdot{\CC}$ that is equivalent to the norm on $X$. Thus, the variational inequality can be put as
\begin{align*}
  \ip{-\lambda^{-1}\CC^{-1}\BB^*\II^\star\pp +\uu}{\vv-\uu}_{\CC}\geq 0 \quad\forall \vv\in \Xad.
\end{align*}
This means that $\uu = \Pi_{\Xad}\lambda^{-1}\CC^{-1}\BB^*\II^\star\pp$ where $\Pi_{\Xad}\colon X\to \Xad$ denotes the projection with respect to the inner product $\ip\cdot\cdot_{\CC}$, see, e.g.~\cite[Theorem~5.2]{Brezis}.
Note that $\Pi_{\Xad}$ is a non linear operator unless $\Xad$ is a closed subspace. In the latter case it is the orthogonal projection. In either case, $\Pi_{\Xad}$ is non-expansive, i.e.,
\begin{align*}
  \norm{\Pi_{\Xad}(\vv)-\Pi_{\Xad}(\ww)}{\CC} \leq \norm{\vv-\ww}{\CC} \quad\forall \vv,\ww\in X. 
\end{align*}
Hence, by norm equivalence we also have that
\begin{align*}
  \norm{\Pi_{\Xad}(\vv)-\Pi_{\Xad}(\ww)}{X} \lesssim \norm{\vv-\ww}{X} \quad\forall \vv,\ww\in X. 
\end{align*}

\begin{theorem}
  Suppose that $1\eqsim \alpha\geq \alpha_0$, $\lambda^{-1}\eqsim \beta\geq\beta_0$, such that the results of Theorem~\ref{thm:constrained} hold true. 
  Let $(\uu,\yy,\pp) \in \Xad\times Y\times Y^\star$ denote the unique solution of~\eqref{eq:varin:constrained}. 
  Let $X_{h,\mathrm{ad}}\subset \Xad$ denote a non-empty closed convex subset and $Y_h\subset Y$, $Y_h^\star\subset Y^\star$ closed subspaces.
  Let $(\uu_h,\yy_h,\pp_h)\in X_{h,\mathrm{ad}}\times Y_h\times Y_h^\star$ be some arbitrary element and define
  \begin{align*}
    \widetilde\uu_h = \widetilde\uu_h(\pp_h) = \Pi_{\Xad} \lambda^{-1}\CC^{-1}\BB^*\II^\star\pp_h.
  \end{align*}
  Then, the estimator given by
  \begin{align*}
    \eta^2 = \alpha\norm{\LL\yy_h+\BB\uu_h-\ff}{Z}^2 + \beta\norm{\LL^\star\pp_h-\AA^*(\AA\II\yy_h-\zz_d)}{Z}^2 
    + \lambda\norm{\widetilde\uu_h-\uu_h}{X}^2
  \end{align*}
  is reliable and efficient in the sense that
  \begin{align*}
    \lambda\norm{(\uu-\uu_h,\yy-\yy_h,\pp-\pp_h)}{X\times Y \times Y^\star} \lesssim\eta &\lesssim \lambda^{-1/2}\norm{(\uu-\uu_h,\yy-\yy_h,\pp-\pp_h)}{X\times Y \times Y^\star}.
  \end{align*}
\end{theorem}
\begin{proof}
  We start by noting that $\widetilde\uu_h \in \Xad$ satisfies
  \begin{align*}
    \ip{\lambda\CC\widetilde\uu_h -\BB^*\II^\star \pp_h}{\vv-\widetilde\uu_h}_X \geq 0 \quad\forall \vv\in \Xad.
  \end{align*}
  First, we show the efficiency estimate. By the triangle inequality, $\LL\yy + \BB\uu=\ff$, $\LL^\star\pp=\AA^*(\AA\II\yy-\zz_d)$,
  estimate~\eqref{eq:proofcoerc1}, Lipschitz continuity of the projection $\Pi_{\Xad}$ and Lemma~\ref{lem:normequiv}, we have that
  \begin{align*}
    \eta^2 &= \alpha\norm{\LL\yy_h+\BB\uu_h-\ff}{Z}^2 + \beta\norm{\LL^\star\pp_h-\AA^*(\AA\II\yy_h-\zz_d)}{Z}^2 
    + \lambda\norm{\widetilde\uu_h-\uu_h}{X}^2
    \\
    &= \alpha\norm{\LL(\yy_h-\yy)+\BB(\uu_h-\uu)}{Z}^2 + \beta\norm{\LL^\star(\pp_h-\pp)-\AA^*\AA\II(\yy_h-\yy)}{Z}^2 
    + \lambda\norm{\widetilde\uu_h-\uu_h}{X}^2
    \\
    &\eqsim \norm{\LL(\yy_h-\yy)+\BB(\uu_h-\uu)}{Z}^2 + \lambda^{-1}\norm{\LL^\star(\pp_h-\pp)-\AA^*\AA\II(\yy_h-\yy)}{Z}^2 + \lambda\norm{\widetilde\uu_h-\uu_h}{X}^2
    \\
    &\lesssim  \enorm{(\uu-\uu_h,\yy-\yy_h,\pp-\pp_h)}_\lambda^2 + \lambda\norm{\uu-\uu_h}{X}^2 + \lambda\norm{\widetilde\uu_h-\uu}{X}^2
    \\
    &\lesssim \lambda^{-1}\norm{(\uu-\uu_h,\yy-\yy_h,\pp-\pp_h)}{X\times Y\times Y^\star}^2
    \\ &\qquad\qquad+ \lambda
    \norm{\Pi_{\Xad}\lambda^{-1}\CC^{-1}\BB^*\II^\star\pp_h-\Pi_{\Xad}\lambda^{-1}\CC^{-1}\BB^*\II^\star\pp}{X}^2
    \\
    &\lesssim \lambda^{-1}\norm{(\uu-\uu_h,\yy-\yy_h,\pp-\pp_h)}{X\times Y\times Y^\star}^2 + \lambda^{-1}\norm{\pp-\pp_h}{Y^\star}^2
    \\&\eqsim \lambda^{-1}\norm{(\uu-\uu_h,\yy-\yy_h,\pp-\pp_h)}{X\times Y\times Y^\star}^2.
  \end{align*}

  For the lower bound, we apply Theorem~\ref{thm:constrained} to see that %
  \begin{align*}
    \lambda \enorm{(\uu-\widetilde\uu_h,\yy-\yy_h,\pp-\pp_h)}_\lambda^2 
    &\lesssim \lambda\norm{\LL(\yy_h-\yy)+\BB(\widetilde\uu_h-\uu)}Z^2 
    \\ 
    &\qquad + \norm{\LL^\star(\pp_h-\pp) -\AA^*\AA\II(\yy_h-\yy)}Z^2 
    \\
    &\qquad  + \lambda \ip{-\BB^*\II^\star(\pp_h-\pp)\lambda\CC(\widetilde\uu_h-\uu)}{\widetilde\uu_h-\uu}_X.
  \end{align*}
  Note that $\ip{-\BB^*\II^\star\pp+\lambda\CC\uu}{\widetilde\uu_h-\uu}_X\geq 0$ since $\widetilde\uu_h\in \Xad$, as well as 
  $\ip{-\BB^*\II^\star\pp_h + \lambda\CC\widetilde\uu_h}{\uu-\widetilde\uu_h}_X\geq 0$ since $\uu\in \Xad$. Thus, 
  \begin{align*}
    \ip{-\BB^*\II^\star(\pp_h-\pp)+\lambda\CC(\widetilde\uu_h-\uu)}{\widetilde\uu_h-\uu}_X  
    \leq \ip{-\BB^*\II^\star\pp_h+\lambda\CC\widetilde\uu_h}{\widetilde\uu_h-\uu}_X \leq 0.
  \end{align*}
  We conclude that 
  \begin{align*}
    \lambda \enorm{(\uu-\widetilde\uu_h,\yy-\yy_h,\pp-\pp_h)}_\lambda^2 &\lesssim \lambda\norm{\LL\yy_h+\BB\widetilde\uu_h-\ff}Z^2
    + \norm{\LL^\star\pp_h-\AA^*(\AA\II\yy-\zz_d)}Z^2
    \\
    &\lesssim \lambda\norm{\LL\yy_h+\BB\uu_h-\ff}Z^2
    + \norm{\LL^\star\pp_h-\AA^*(\AA\II\yy-\zz_d)}Z^2 
    \\&\qquad+ \lambda\norm{\widetilde\uu_h-\uu_h}X^2
    \lesssim \eta^2. 
  \end{align*}%
  The triangle inequality $\norm{\uu-\uu_h}{X}\leq \norm{\widetilde\uu_h-\uu_h}{X} + \norm{\uu-\widetilde\uu_h}{X}$ 
  and Lemma~\ref{lem:normequiv} then show %
  \begin{align*}
    \lambda^2\norm{(\uu-\uu_h,\yy-\yy_h,\pp-\pp_h)}{X\times Y\times Y^\star}^2 \lesssim \lambda \enorm{(\uu-\widetilde\uu_h,\yy-\yy_h,\pp-\pp_h)}_\lambda^2 \lesssim \eta^2,
  \end{align*}
which finishes the proof.
\end{proof}

For general space $X$ and non-empty convex closed sets $\Xad$ the evaluation of $\Pi_{\Xad}$ may not be easily accessible.
However, for the widely considered case of $X = \prod_{j=1}^m \boldsymbol{L}^2(\Omega_j')$ ($\Omega_1',\dots,\Omega_m'$ Lipschitz domains) together with box constraints 
there is an explicit representation of $\Pi_{\Xad}$:
Let $\aa,\bb$ be given vector-valued and essentially bounded functions,
$\aa<\bb$ (understood componentwise a.e.) and consider the box constraint control set
\begin{align*}
  \Xad^\mathrm{box} = \set{\vv\in X}{\aa\leq \vv \leq \bb \text{ a.e.}}.
\end{align*}
Suppose that $\CC$ is the identity operator. The projection on $\Xad^\mathrm{box}$,
characterized by $\ip{\Pi_{\Xad^\mathrm{box}}\vv-\vv}{\vv-\ww}\geq 0$ for all $\ww\in \Xad^\mathrm{box}$, has the explicit representation, see, e.g.,~\cite[Theorem~2.28]{Troeltzsch},
\begin{align*}
  \Pi_{\Xad^\mathrm{box}} \vv = \min\{\bb,\max\{\vv,\aa\}\} \quad\text{a.e.}
\end{align*}
The latter definition is understood componentwise.

\subsection{LSFEM for unconstrained optimal control problem}\label{sec:lsfem:unconstrained}
In this section we consider the unconstrained optimal control problem, i.e, $\Xad\subseteq X$ is a closed subspace.
This allows us to eliminate the control variable:
The variational inequality~\eqref{eq:model:c} is
\begin{align*}
  \ip{-\BB^*\II^\star\pp+\lambda\CC\uu}{\vv-\uu}_{X} &\geq 0 \quad\forall \vv\in \Xad.
\end{align*}
Taking $\vv = \ww+ \uu$ and $\vv=\uu-\ww$ with $\ww\in \Xad$ implies
\begin{align*}
  \ip{-\BB^*\II^\star\pp+\lambda\CC\uu}{\ww}_{X} &= 0 \quad\forall \ww\in \Xad.
\end{align*}
Therefore, $\uu = \Pi_{\Xad}\lambda^{-1}\CC^{-1}\BB^*\II^\star\pp$ where $\Pi_{\Xad}$ is the orthogonal projection with respect to the inner product $\ip{\cdot}\cdot_{\CC}$.
Problem~\eqref{eq:model} simplifies to
\begin{align*}
  \LL\yy &= \ff-\BB\Pi_{\Xad}\lambda^{-1}\CC^{-1}\BB^*\II^\star\pp, \\
  \LL^\star\pp &= \AA^*(\AA\II\yy-\zz_d).
\end{align*}
Define the least-squares functional
\begin{align*}
  G(\yy,\pp;\ff,\zz_d) = \alpha\norm{\LL\yy+\BB\Pi_{\Xad}\lambda^{-1}\CC^{-1}\BB^*\II^\star\pp-\ff}{Z}^2 + \beta\norm{\LL^\star\pp-\AA^*(\AA\II\yy-\zz_d)}{Z}^2,
\end{align*}
where $\alpha,\beta>0$, 
and consider the minimization problem
\begin{align}\label{eq:minUnconstrained}
  \min_{(\yy,\pp)\in Y\times Y^\star} G(\yy,\pp;\ff,\zz_d).
\end{align}
This problem has a unique solution if we can prove that the functional $(\yy,\pp)\mapsto G(\yy,\pp;0,0)$ is equivalent to a norm on $Y\times Y^\star$, see, e.g.,~\cite[Section~2.2.1 and Theorem~2.5]{BochevGunzberger09}.
We use the norm
\begin{align*}
  \enorm{(\yy,\pp)}_\lambda^2 := \enorm{(\Pi_{\Xad}\lambda^{-1}\CC^{-1}\BB^*\II^\star\pp,\yy,\pp)}_\lambda^2 \quad\forall (\yy,\pp)\in Y\times Y^\star.
\end{align*}
By following the proof of Lemma~\ref{lem:normequiv} one verifies the next result.
\begin{lemma}\label{lem:normequiv2}
  The estimates
  \begin{align*}
    \lambda^{1/2}\norm{(\yy,\pp)}{Y\times Y^\star} \lesssim \enorm{(\yy,\pp)}_\lambda \lesssim \lambda^{-1/2}\norm{\yy}{Y} + \lambda^{-1}\norm{\pp}{Y^\star}
  \end{align*}
  hold for all $(\yy,\pp)\in Y\times Y^\star$.
\end{lemma}
\begin{proof}
  The lower bound directly follows from the definitions and Lemma~\ref{lem:normequiv}. 
  For the upper bound we use $\uu = \Pi_{\Xad}\lambda^{-1}\CC^{-1}\BB^*\II^\star\pp$. In the proof of Lemma~\ref{lem:normequiv} it was shown that
  \begin{align*}
    \enorm{(\uu,\yy,\pp)}_\lambda^2 &\lesssim \norm{\uu}{X}^2 + \lambda^{-1}\norm{\yy}{Y}^2 + \lambda^{-1} \norm{\pp}{Y^\star}^2. 
  \end{align*}
  This together with $\norm{\uu}X\lesssim \lambda^{-1}\norm{\pp}{Y^\star}$ shows the upper bound.
\end{proof}

\begin{theorem}\label{thm:unconstrained}
  If $\alpha\eqsim 1$, $\beta\eqsim\lambda^{-1}$, then 
  \begin{align*}
    \enorm{(\yy,\pp)}_{\lambda}^2\eqsim G(\yy,\pp;0,0)  \quad\forall (\yy,\pp)\in Y\times Y^\star.
  \end{align*}
\end{theorem}
\begin{proof}
  Let $\alpha_0\eqsim 1$, $\beta_0\eqsim \lambda^{-1}$ denote the constants from Theorem~\ref{thm:constrained}.
  Note that we assume $\alpha\eqsim 1$ and $\beta\eqsim \lambda^{-1}$. Therefore, $\alpha\eqsim \alpha_0$, $\beta\eqsim \beta_0$.
  Let $\yy\in Y$ and $\pp\in Y^\star$ be given. Define $\uu = \Pi_{\Xad}\lambda^{-1}\CC^{-1}\BB^*\II^\star\pp$. 
  Note that then $\ip{-\BB^*\II^\star\pp+\lambda\CC\uu}{\vv}_X=0$ for any $\vv\in \Xad$. 
  We have that
  \begin{align*}
    G(\yy,\pp;0,0) &= \alpha\norm{\LL\yy+\BB\Pi_{{\Xad}}\lambda^{-1}\CC^{-1}\BB^*\II^\star\pp}{Z}^2 + \beta\norm{\LL^\star\pp-\AA^*\AA\II\yy}{Z}^2
    \\
    &\eqsim \alpha_0 \norm{\LL\yy+\BB\Pi_{{\Xad}}\lambda^{-1}\CC^{-1}\BB^*\II^\star\pp}{Z}^2 + \beta_0\norm{\LL^\star\pp-\AA^*\AA\II\yy}{Z}^2
    \\
    &= \alpha_0\norm{\LL\yy+\BB\uu}{Z}^2 + \beta_0\norm{\LL^\star\pp-\AA^*\AA\II\yy}{Z}^2 
    + \ip{-\BB^*\II^\star\pp+\lambda\CC\uu}{\vv}_X
    \\ &\eqsim \enorm{(\uu,\yy,\pp)}_\lambda^2 = \enorm{(\yy,\pp)}_\lambda^2.
  \end{align*}
  The last equivalence ``$\eqsim$'' follows from the proof of Theorem~\ref{thm:constrained}.
\end{proof}

The variational formulation of the convex minimization problem~\eqref{eq:minUnconstrained} reads: Find $(\yy,\pp)\in Y\times Y^\star$ such that
\begin{align}\label{eq:varForm:unconstrained}
  b(\yy,\pp;\ww,\qq) = m(\ww,\qq) \quad\forall (\ww,\qq)\in Y\times Y^\star, 
\end{align}
where the bilinear form $b\colon (Y\times Y^\star)^2\to \R$ and the functional $m\colon Y\times Y^\star\to\R$ are given by
\begin{align*}
  b(\yy,\pp;\ww,\qq) &= \alpha \ip{\LL\yy+\BB\Pi_{\Xad}\lambda^{-1}\CC^{-1}\BB^*\II^\star\pp}{\LL\ww+\BB\Pi_{\Xad}\lambda^{-1}\CC^{-1}\BB^*\II^\star\qq}_Z
  \\
  &\qquad + \beta\ip{\LL^\star\pp-\AA^*\AA\II\yy}{\LL^\star\qq-\AA^*\AA\II\ww}_Z
  \\
  m(\ww,\qq) &= \alpha\ip{\ff}{\LL\ww+\BB\Pi_{\Xad}\lambda^{-1}\CC^{-1}\BB^*\II^\star\qq}_Z + \beta\ip{-\AA^*\zz_d}{\LL^\star\qq-\AA^*\AA\II\ww}_Z.
\end{align*}
Given a closed subspace $Y_h\times Y_h^\star\subseteq Y\times Y^\star$ we consider the discrete formulation: 
Find $(\yy_h,\pp_h)\in Y_h\times Y_h^\star$ such that
\begin{align}\label{eq:varForm:unconstrained:discrete}
  b(\yy_h,\pp_h;\ww,\qq) = m(\ww,\qq) \quad\forall (\ww,\qq)\in Y_h\times Y_h^\star.
\end{align}

Thanks to the norm equivalence established in Theorem~\ref{thm:unconstrained} the next result follows immediately from the theory of LSFEMs, see, e.g.,~\cite[Chapter~3 and Theorem~2.5]{BochevGunzberger09}.
\begin{theorem}\label{thm:optimality:unconstrained}
  Let $\ff\in Z$, $\zz_d\in Z$ be given. 
  The problems~\eqref{eq:varForm:unconstrained} and~\eqref{eq:varForm:unconstrained:discrete} admit unique solutions $(\yy,\pp)\in Y\times Y^\star$ and $(\yy_h,\pp_h)\in Y_h\times Y_h^\star$. 
  They satisfy the quasi-optimality estimate %
  \begin{align*}
    \enorm{(\yy-\yy_h,\pp-\pp_h)}_\lambda \lesssim \inf_{(\zz_h,\qq_h)\in Y_h\times Y_h^\star} \enorm{(\yy-\zz_h,\pp-\qq_h)}_\lambda.
  \end{align*}%
\end{theorem}

Another immediate result of the norm equivalence is the following a posteriori error estimate.
\begin{corollary}
  Let $(\yy,\pp)\in Y\times Y^\star$ denote the unique solution of~\eqref{eq:varForm:unconstrained} and let $Y_h\times Y_h^\star\subset Y\times Y^\star$ denote a closed subspace. 
  The least-squares functional defines an efficient and reliable a posteriori error estimator, i.e., 
  \begin{align*}
    \enorm{(\yy-\yy_h,\pp-\pp_h)}_\lambda^2 \eqsim G(\yy_h,\pp_h;\ff,\zz_d) \quad\text{for any } (\yy_h,\pp_h)\in Y_h\times Y_h^\star.
  \end{align*}
\end{corollary}

\begin{remark}
  To obtain an approximation of the control $\uu$ given that $(\yy_h,\pp_h)\in Y_h\times Y_h^\star$ is the solution of~\eqref{eq:varForm:unconstrained:discrete}, we define
  \begin{align*}
    \uu_h = \Pi_{\Xad}\lambda^{-1}\CC^{-1}\BB^*\II^\star\pp_h.
  \end{align*}
  Since $\uu = \Pi_{\Xad}\lambda^{-1}\CC^{-1} \BB^*\II^\star\pp$ we get that
  \begin{align*}
    \lambda^{1/2}\norm{\uu-\uu_h}{X} \leq \enorm{(\yy-\yy_h,\pp-\pp_h)}_\lambda \lesssim \inf_{(\zz_h,\qq_h)\in Y_h\times Y_h^\star} \enorm{(\yy-\zz_h,\pp-\qq_h)}_\lambda
  \end{align*}
  by the boundedness of the involved operators and Theorem~\ref{thm:optimality:unconstrained}.
\end{remark}

\section{Examples}\label{sec:examples}
In this section we present various model problems where the framework developed in Section~\ref{sec:main} can be applied. 

\subsection{Second-order PDE}
Let $\Omega\subseteq \R^d$, $d\geq 2$ denote a bounded Lipschitz domain. 
Given $f\in L^2(\Omega;\R)$, $\ff\in L^2(\Omega;\R^d)$ we consider the first-order formulation of a general second-order PDE, 
\begin{subequations}\label{eq:secondOrderPDE}
  \begin{align}
    \div\ssigma + cy &= f, \\
    \nabla y - \bb y + \Amat^{-1}\ssigma &= \ff, \\
    y|_{\partial \Omega} &= 0,
  \end{align}
\end{subequations}
where $\Amat\in L^\infty(\Omega;\R^{d\times d})$ is a symmetric, uniformly positive matrix, $\bb\in L^\infty(\Omega;\R^d)$, $c\in L^\infty(\Omega;\R)$ such that
\begin{align*}
  L^\infty(\Omega)\ni\frac12\div(\Amat\bb) + c \geq 0.
\end{align*}
We can also include the Helmholtz problem, where $\Amat$ is the identity, $\bb=0$, and $c<0$ is a constant such that it is not an eigenvalue of
\begin{align*}
  -\Delta y = \mu y, \quad y|_{\partial\Omega} = 0. 
\end{align*}

We note that the problem admits a unique solution $(y,\ssigma)\in H_0^1(\Omega)\times \Hdivset\Omega$. 
This can be seen by solving the second equation in~\eqref{eq:secondOrderPDE} for $\ssigma$ and then replacing $\ssigma$ in the first equation. 
The resulting second-order PDE is
\begin{align*}
  -\div(\Amat\nabla y) + \div(\Amat\bb y) + cy = f -\div(\Amat\ff), \quad y|_{\partial\Omega} = 0.
\end{align*}
Unique solvability can be shown by standard arguments and we omit the details here.

The choices
\begin{itemize}
  \item $X = H=L^2(\Omega;\R)$, $\Xad = \set{v\in X}{a\leq v\leq b \text{ a.e. in }\Omega}$ with $a<b$,
  \item $Y = Y^\star = H_0^1(\Omega)\times \Hdivset\Omega$,
  \item $Z = L^2(\Omega;\R)\times L^2(\Omega;\R^d)$
  \item $\II(y,\ssigma) = (y,\ssigma)$, $\II^\star(p,\xxi) = (p,\xxi)$,
  \item $\AA(y,\ssigma) = y$,
  \item $\BB u = -(u,0)$,
  \item $\CC u = \lambda u$ with $\lambda>0$,
\end{itemize}
give rise to the following optimal control problem: Given $z_d\in L^2(\Omega;\R)$ and $(f,\ff)\in Z$ find the minimizer of
\begin{align*}
  \min \norm{y-z_d}{L^2(\Omega)}^2 + &\lambda\norm{u}{L^2(\Omega)}^2 \quad\text{subject to}\\
    \div\ssigma + cy &= f + u, \\
    \nabla y - \bb y + \Amat^{-1}\ssigma &= \ff, \\
    y|_{\partial \Omega} &= 0,
\end{align*}
The corresponding least-squares operators are given by
\begin{align*}
  \LL\yy &= \left(\div\ssigma + c y, \nabla y - \bb y + \Amat^{-1}\ssigma\right),  \quad
  \LL^\star\pp = \left(-\div\xxi - \bb\cdot\xxi + c p, -\nabla p + \Amat^{-1}\xxi\right)
\end{align*}
for $\yy = (y,\ssigma)\in Y$, $\pp=(p,\xxi)\in Y^\star$. In particular, integration by parts proves that
\begin{align*}
  \ip{\LL\yy}{\II^\star\pp} = \ip{\II\yy}{\LL^\star\pp} \quad\forall \yy\in Y, \pp\in Y^\star.
\end{align*}
We note that $\ran(\LL) = Z = \ran(\LL^\star)$, $\ker(\LL)=\{0\} = \ker(\LL^\star)$ which can be seen by following the analysis of, e.g.,~\cite{CaiLazarovManteuffelMcCormickPart1}, so that this example fits the abstract framework from Section~\ref{sec:main}. 

\subsection{Stokes problem}
Let $\Omega\subseteq \R^d$, $d=2,3$ denote a bounded Lipschitz domain. We consider the Stokes equations
\begin{align*}
  -\boldsymbol\Delta \yy + \nabla p &= \ff, \\
  \div\yy &= 0, \\
  \yy|_{\partial \Omega} &=0.
\end{align*}
For a suitable least-squares formulation we use the pseudostress tensor $\MM = \boldsymbol\nabla \yy - p\Imat$, see~\cite[Section~3.2]{CaiLeeWang04} for the definition and analysis of a similar formulation, where $\Imat$ denotes the identity tensor.
Here, the vector gradient $\boldsymbol\nabla\vv$ is the tensor where the $j$-th row is the gradient of the $j$-th component of $\vv$ and $\boldsymbol\Delta\vv$ is the tensor where the $j$-th row is the Laplacian of the $j$-th component of $\vv$.
In particular, by noting that $\tr\MM = \div\yy-d p = -d p$ we can eliminate the pressure variable $p$ from the system. Let $\Dev\MM = \MM-\tfrac1d\Imat\tr\MM$ denote the deviatoric part of $\MM\in L^2(\Omega;\R^{d\times d})$, $\Pi_\Omega$ the $L^2(\Omega;\R)$ orthogonal projection onto constants, and $\Div$ the row-wise divergence operator. 
Using the spaces 
\begin{align*}
  Y = Y^\star = \HH_0^1(\Omega)\times \HDivset\Omega = \HH_0^1(\Omega)\times \set{\NN\in L^2(\Omega;\R^{d\times d})}{\Div\NN\in L^2(\Omega;\R^d)}
\end{align*}
we define the least-squares operators
\begin{align}\label{eq:stokes:op}
  \LL(\yy,\MM) = \LL^\star(\yy,\MM) = (-\Div\MM,\boldsymbol\nabla\yy-\Dev\MM - \frac1d\Imat\Pi_\Omega\tr\MM),
\end{align}
and $\II(\yy,\MM) = \II^\star(\yy,\MM) = (\yy,\MM)$. One verifies that these operators are linear and bounded in the canonical norms.
We note that $\ip{\MM}{\NN}_{L^2(\Omega)} = \int_\Omega \MM:\NN\,\di x$, where $\MM:\NN$ denotes the Frobenius inner product of two tensors $\MM,\NN \in \HDivset\Omega$. Then, $\ip{\Dev\MM}{\Imat}_{L^2(\Omega)} = 0$ and $\ip{\Dev\MM}{\NN}_{L^2(\Omega)} = \ip{\MM}{\Dev\NN}_{L^2(\Omega)}$ by a straightforward calculation.
To see~\eqref{ass:I} set $Z = L^2(\Omega;\R^d)\times L^2(\Omega;\R^{d\times d})$ and 
let $(\yy,\MM),(\pp,\NN)\in Y=Y^\star$ be given. Integration by parts yields
\begin{align*}
  \ip{\LL(\yy,\MM)}{\II^\star(\pp,\NN)}_Z &= \ip{-\div\MM}{\pp}_{L^2(\Omega)} + \ip{\boldsymbol\nabla\yy-\Dev\MM - \frac1d\Imat\Pi_\Omega\tr\MM}{\NN}_{L^2(\Omega)}
  \\
  &= \ip{\MM}{\boldsymbol\nabla \pp}_{L^2(\Omega)} + \ip{\yy}{-\Div\NN}_{L^2(\Omega)} + \ip{\MM}{-\Dev\MM}_{L^2(\Omega)} 
  \\
  &\qquad - \frac1d \ip{\Pi_\Omega\tr\MM}{\tr\NN}_{L^2(\Omega)}
  \\
  &= \ip{\yy}{-\Div\NN}_{L^2(\Omega)} + \ip{\MM}{\boldsymbol\nabla\pp-\Dev\NN-\frac1d\Imat\Pi_\Omega\tr\NN}_{L^2(\Omega)}
  \\
  &= \ip{\II(\yy,\MM)}{\LL^\star(\pp,\NN)}_Z.
\end{align*}

Assumption~\eqref{ass:L} is shown in Theorem~\ref{thm:stokes} below. 
For the optimal control problem of interest we further have
\begin{itemize}
  \item $X = H = L^2(\Omega;\R^d)$, $\Xad = \set{\vv\in X}{\aa\leq \vv\leq \bb \text{ a.e.}}$ with $\aa,\bb \in \R^d$,
  \item $\AA(\yy,\MM) = \yy$, $\BB\uu = -(\uu,0)$, $\CC\uu = \uu$.
\end{itemize}
Thus, we have the optimal control problem: Given $(\ff,0)\in Z$, $\zz_d\in H$ find the solution $\uu\in \Xad$ of
\begin{align}\label{eq:stokes}
\begin{split}
  \min \norm{\yy-\zz_d}{L^2(\Omega)}^2 + &\lambda \norm{\uu}{L^2(\Omega)}^2 \text{ subject to}
  \\
  -\Div\MM &= \ff+\uu, \\
  \boldsymbol\nabla \yy - \Dev\MM - \frac1d\Imat\Pi_\Omega\tr\MM &= 0, \\
  \yy|_{\partial \Omega} &= 0. 
\end{split}
\end{align}
We stress that the equations above are a formulation of the Stokes problem 
\begin{align*}
  -\boldsymbol\Delta \yy + \nabla p &= \ff+\uu, \quad \div\yy = 0, \quad\text{and } \Pi_\Omega p = 0.
\end{align*}

The proof of the next theorem uses unique solvability of the weak formulation of the Stokes problem, see, e.g.,~\cite[Theorem~8.2.1]{BoffiBrezziFortin}.
\begin{theorem}\label{thm:stokes}
  Let $\LL$ denote the operator from~\eqref{eq:stokes:op}. We have that $\ker\LL = \{0\}$ and $\ran\LL = Z$. 
\end{theorem}
\begin{proof}
  Let $(\ff,\FF)\in Z$ be given. We prove surjectivity of $\LL$. To that end, let $(\yy,p)\in \HH_0^1(\Omega)\times L^2(\Omega;\R)$ with $\Pi_\Omega p=0$ denote the (unique) weak solution of the Stokes equations
  \begin{align*}
    -\Delta \yy + \nabla p &= \ff-\Div\FF
    \\
    \div\yy &= (1-\Pi_\Omega)\tr\FF.
  \end{align*}
  Then, define $\MM:= \boldsymbol\nabla\yy-p\Imat -\FF\in L^2(\Omega;\R^{d\times d})$. Taking the trace gives
  \begin{align*}
    \tr\MM = \div\yy-d p -\tr\FF = (1-\Pi_\Omega)\tr\FF-dp-\tr\FF = -\Pi_\Omega\tr\FF-dp.
  \end{align*}
  Applying $\Pi_\Omega$ we get $\Pi_\Omega\tr\MM = -\Pi_\Omega\tr\FF$. Putting the last identities together we see that
  \begin{align*}
    \Dev\MM = \boldsymbol\nabla\yy-p\Imat-\FF - \frac1d\Imat\Pi_\Omega\tr\MM + p\Imat,
  \end{align*}
  or equivalently, $\boldsymbol\nabla\yy-\Dev\MM -\tfrac1d\Imat\Pi_\Omega\tr\MM = \FF$.
  We have to show that $\Div\MM\in L^2(\Omega;\R^d)$ and $\Div\MM=-\ff$. 
  Recall that from the weak formulation of the Stokes problem we have that
  \begin{align*}
    \ip{\boldsymbol\nabla \yy}{\boldsymbol\nabla\vv}_{L^2(\Omega)} -\ip{p}{\div\vv}_{L^2(\Omega)} = \ip{\ff}{\vv}_{L^2(\Omega)}+\ip{\FF}{\boldsymbol\nabla \vv}_{L^2(\Omega)} \quad\forall \vv\in \HH_0^1(\Omega).
  \end{align*}
  Using this identity we get for any $\vv \in \DD(\Omega)\subseteq \HH_0^1(\Omega)$ (smooth test functions with vanishing traces)
  \begin{align*}
    \dual{\Div\MM}{\vv} &= -\ip{\MM}{\boldsymbol\nabla\vv}_{L^2(\Omega)} = -\ip{\boldsymbol\nabla\yy}{\boldsymbol\nabla\vv}_{L^2(\Omega)} + \ip{p}{\div\vv}_{L^2(\Omega)} + \ip{\FF}{\boldsymbol\nabla \vv}_{L^2(\Omega)} 
    \\
    &= \ip{-\ff}{\vv}_{L^2(\Omega)}.
  \end{align*}
  Due to density of $\DD(\Omega)$ in $L^2(\Omega;\R^d)$ this proves that $\Div\MM = -\ff$. Thus, we have proven that for given $(\ff,\FF)\in Z$ there exists $(\yy,\MM)\in Y$ with $\LL(\yy,\MM) = (\ff,\FF)$.

  To show injectivity of $\LL$ let $(\yy,\MM)\in Y$ with $\LL(\yy,\MM)=0$ be given, i.e.,
  \begin{align*}
    -\Div\MM &= 0, \\
    \boldsymbol\nabla\yy -\Dev\MM - \frac1d\Imat\Pi_\Omega\tr\MM &= 0. 
  \end{align*}
  Taking the trace and then applying $\Pi_\Omega$ to the second equation shows $\Pi_\Omega\tr\MM = 0$. We then have
  \begin{align*}
    \boldsymbol\nabla\yy-\Dev\MM = 0.
  \end{align*}
  Taking the trace of the latter identity we see that $\div\yy = 0$. Setting $p:=-\tfrac1d\tr\MM$ and eliminating $\MM$ in the equation $\Div\MM = 0$ we see that $(\yy,p)$ satisfy the Stokes equations
  \begin{align*}
    -\Delta \yy + \nabla p &=0, \quad \div\yy = 0, \quad \Pi_\Omega p = 0. 
  \end{align*}
  By unique solvability we conclude that $\yy = 0$. Furthermore, we conclude $\Dev\MM = 0$, that is $\MM = \tfrac1d\Imat\tr\MM = -p\Imat = 0$. This finishes the proof.
\end{proof}

\subsection{Maxwell equation}
We consider the real 3D Maxwell problem: 
Let $\Omega$ denote a simply connected Lipschitz domain.
Given $\ff_1,\ff_2\in L^2(\Omega;\R^3)$, $c\in L^\infty(\Omega;\R)$, find $(\yy,\ssigma)\in \HcurlsetZero\Omega\times \Hcurlset\Omega$ such that
\begin{align}
\begin{split}\label{eq:maxwell}
  \ccurl\ssigma + c\yy &= \ff_1, \\
  \ccurl\yy - \ssigma &= \ff_2.
\end{split}
\end{align}
This is a first-order reformulation of the PDE
\begin{align}\label{eq:maxwellPDE}
  \ccurl\ccurl\yy + c\yy = \ff_1 + \ccurl\ff_2, \quad \yy\times\normal|_{\partial \Omega} = 0. 
\end{align}
Here, the spaces are given by
\begin{align*}
  \Hcurlset\Omega &:= \set{\vv\in L^2(\Omega;\R^3)}{\ccurl\vv \in L^2(\Omega;\R^3)}, 
  \quad \\
  \HcurlsetZero\Omega &= \set{\vv\in \Hcurlset\Omega}{\vv\times\normal|_{\partial \Omega} = 0},
\end{align*}
where $\normal$ denotes the normal vector field on $\partial\Omega$. Moreover, we assume that $c$ is strictly positive, or if $c<0$, that $c$ is a constant and is not an eigenvalue of the cavity problem, see, e.g.~\cite[Section~2.4]{CCStorn18} and~\cite[Section~4.7]{Monk03}.
We note that $\II(\yy,\ssigma) = (\yy,\ssigma)$, $\II^\star=\II$, 
\begin{align*}
  \LL(\yy,\ssigma) = \left( \ccurl\ssigma + c\yy, \ccurl\yy-\ssigma\right)
  \quad\forall (\yy,\ssigma)\in Y =Y^\star := \HcurlsetZero\Omega\times\Hcurlset\Omega,
\end{align*}
$\LL^\star=\LL$, and $Z=L^2(\Omega;\R^3)\times L^2(\Omega;\R^3)$.
To see~\eqref{ass:I} integration by parts yields for all $(\yy,\ssigma),(\zz,\ttau)\in Y$,
\begin{align*}
  \ip{\LL(\yy,\ssigma)}{\II^\star(\zz,\ttau)}_Z &= \ip{\ccurl\ssigma + c\yy}{\zz}_{L^2(\Omega)} + \ip{\ccurl\yy - \ssigma}{\ttau}_{L^2(\Omega)}
  \\
  &= \ip{\ssigma}{\ccurl\zz}_{L^2(\Omega)} + \ip{\yy}{c\zz}_{L^2(\Omega)} + \ip{\yy}{\ccurl\ttau}_{L^2(\Omega)} - \ip{\ssigma}{\ttau}_{L^2(\Omega)}
  \\
  &= \ip{\yy}{\ccurl\ttau + c\zz}_{L^2(\Omega)} + \ip{\ssigma}{\ccurl\zz-\ttau}_{L^2(\Omega)} 
  \\ &= \ip{\II(\yy,\ssigma)}{\LL^\star(\zz,\ttau)}_{Z}.
\end{align*}
Surjectivity resp. injectivity of $\LL$ can be seen by analyzing problem~\eqref{eq:maxwell} or, equivalently,~\eqref{eq:maxwellPDE}.
In the case $\ff_2=0$ problem~\eqref{eq:maxwellPDE} has been studied in full length in~\cite[Ch.4]{Monk03}. 
There it is also mentioned that general right-hand side ($\ff_2\neq 0$ in our setting) can be included.
We provide a proof for completeness. 
\begin{proposition}
  We have that $\ran(\LL)=Z$ and $\ker(\LL)=\{0\}$.
\end{proposition}
\begin{proof}
  It is straightforward to verify that~\eqref{eq:maxwellPDE} and~\eqref{eq:maxwell} are equivalent. Thus, it suffices to analyze the weak form of~\eqref{eq:maxwellPDE}. 
  The variational formulation of~\eqref{eq:maxwellPDE} is: Find $\yy\in \HcurlsetZero\Omega$ such that
  \begin{align}\label{eq:maxwellVar}
    \ip{\ccurl\yy}{\ccurl\zz}_{L^2(\Omega)} + \ip{c\yy}{\zz}_{L^2(\Omega)} = \ip{\ff_1}{\zz}_{L^2(\Omega)} + \ip{\ff_2}{\ccurl\zz}_{L^2(\Omega)}
  \end{align}
  for all $\zz\in \HcurlsetZero\Omega$. 
  If $c>0$, then the left-hand side defines a symmetric bilinear form that induces an inner product that is equivalent to the canonic inner product in $\HcurlsetZero\Omega$.
  We may thus assume that $c<0$ is a constant and not an eigenvalue of the cavity problem.
  The case $\ff_2=0$ is found in~\cite[Corollary~4.19]{Monk03}. Note that~\cite[Corollary~4.19]{Monk03} also implies that $\ker(\LL)$ is trivial. 
  By linearity we may thus assume that $\ff_1=0$ and $\ff_2\neq 0$. 
  Consider the Helmholtz decomposition
  \begin{align*}
    \yy = \yy_0 + \nabla y_0 \text{ where } y_0 \in H_0^1(\Omega) \text{ solves } \Delta y_0 = \div\yy.
  \end{align*}
  By construction, $\div \yy_0 = \div\yy-\Delta y_0 = 0$, and $\nabla y_0\times \normal|_{\partial \Omega} = 0$ since $y_0|_{\partial \Omega} = 0$. 
  Then, $\yy_0 = \yy-\nabla y_0 \in \widetilde\HH = \set{\vv\in\HcurlsetZero\Omega}{\div\vv=0}$. 
  Using $\zz = \nabla q$ for any $q\in H_0^1(\Omega)$ in~\eqref{eq:maxwellVar} one finds that $\div\yy=0$, thus, a solution to~\eqref{eq:maxwellVar} satisfies $\yy =\yy_0$. Consequently,~\eqref{eq:maxwellVar} is equivalent to: Find $\yy_0\in\widetilde\HH$ such that
  \begin{align*}
    \ip{\ccurl\yy_0}{\ccurl\zz_0}_{L^2(\Omega)} + c\ip{\yy_0}{\zz_0}_{L^2(\Omega)} = \ip{\ff_2}{\ccurl\zz_0}_{L^2(\Omega)}
  \end{align*}
  for all $\zz_0\in \widetilde\HH$. This variational formulation has a unique solution because the first term on the left-hand side defines a bounded and coercive bilinear form~\cite[Corollary~3.51]{Monk03} and the embedding from $\widetilde\HH$ to $L^2(\Omega,\R^3)$ is compact~\cite[Corollary~3.49]{Monk03}.
\end{proof}

The optimal control problem under consideration is: Given $\ff\in L^2(\Omega;\R^3)$ and $\zz_d\in L^2(\Omega;\R^3)$, solve the minimization problem
\begin{align*}
  \min \norm{\yy-\zz_d}{L^2(\Omega)}^2 + &\lambda \norm{\uu}{L^2(\Omega)}^2 \quad\text{subject to}\\
  \ccurl\ssigma + c\yy &= \ff + \uu, \\
  \ccurl\yy - \ssigma &= 0, \\
  \yy\times\normal|_{\partial \Omega} &= 0. 
\end{align*}
The remaining operators and spaces to fit the abstract framework are given by
\begin{itemize}
  \item $X = H= L^2(\Omega;\R^3)$, $\Xad = \set{\vv\in X}{\boldsymbol{a}\leq \vv\leq \boldsymbol{b} \text{ a.e. }\Omega}$, $\boldsymbol{a}<\boldsymbol{b}$.
  \item $\AA(\yy,\ssigma) = \yy$, $\BB\uu = -(\uu,0)$, $\CC\uu = \uu$.
\end{itemize}

\subsection{Heat equation}\label{sec:examples:heat}
Let $\Omega\subseteq \R^d$, $d\geq 1$ denote a bounded Lipschitz domain, $J =(0,T)$ a time interval, and $Q=J\times \Omega$ the time-space domain.
Given the data $f\in L^2(Q;\R)$, $y_0\in L^2(\Omega;\R)$ consider the heat equation
\begin{align*}
  \partial_t y - \Delta y &=f, \\
  y|_{J\times\partial\Omega} &= 0, \\
  y(0) &= y_0,
\end{align*}
and its first-order reformulation
\begin{align*}
  \partial_t y - \div\ssigma &=f, \\
  \nabla y - \ssigma &= 0, \\
  y|_{J\times\partial\Omega} &= 0, \\
  y(0) &= y_0.
\end{align*}

We analyze the following optimal control problem. Given $f\in L^2(Q;\R)$, $y_0\in L^2(\Omega;\R)$
and the desired states $z_d\in L^2(Q;\R)$, $z_{d,T}\in L^2(\Omega;\R)$, we consider
\begin{align}\label{eq:heat}
\begin{split}
  \min \norm{y-z_d}{L^2(Q)}^2 + \norm{y(T)-z_{d,T}}{L^2(\Omega)}^2 + &\lambda \norm{u}{L^2(Q)}^2 + \lambda \norm{u_0}{L^2(\Omega)}^2 \quad\text{subject to} \\
  \partial_t y -\div\ssigma &= f + u, \\
  \nabla y - \ssigma &= 0, \\
  y|_{J\times\partial \Omega} &= 0, \\
  y(0) &= y_0+u_0.
\end{split}
\end{align}

The spaces and operators that fit the abstract framework are: 
\begin{itemize}
  \item $X = H=L^2(Q;\R)\times L^2(\Omega;\R)$, 
  \item $\Xad = \set{(v,v_0)\in X}{a\leq v \leq b \text{ a.e. in }Q, a_0\leq v_0\leq b_0 \text{ a.e. in }\Omega}$
  \item $Y = \set{(y,\ssigma)\in L^2(H_0^1(\Omega))\times L^2(Q;\R^d)}{\partial_t y-\div\ssigma\in L^2(Q;\R)}$
  \item $Y^\star = \set{(y,\ssigma)\in L^2(H_0^1(\Omega))\times L^2(Q;\R^d)}{\partial_t y+\div\ssigma\in L^2(Q;\R)}$,
  \item $Z = L^2(Q;\R)\times L^2(Q;\R^d) \times L^2(\Omega;\R)$,
  \item $\II(y,\ssigma) = (y,\ssigma,y(T))$, $\II^\star(p,\xxi) = (p,\xxi,p(0))$,
  \item $\AA(y,\ssigma,z) = (y,z)$ for $(y,\ssigma,z)\in Z$,
  \item $\BB (u,u_0) = -(u,0,u_0)$, 
  \item $\CC (u,u_0) = (u,u_0)$.
\end{itemize}
From the analysis of~\cite[Section~2]{GS21} we note that $Y$ and $Y^\star$ are Hilbert spaces with norms
\begin{align*}
  \norm{(y,\ssigma)}{Y}^2 &= \norm{\nabla y}{L^2(Q)}^2 + \norm{\ssigma}{L^2(Q)}^2 + \norm{\partial_ty -\div\ssigma}{L^2(Q)}^2, \\
  \norm{(y,\ssigma)}{Y^\star}^2 &= \norm{\nabla y}{L^2(Q)}^2 + \norm{\ssigma}{L^2(Q)}^2 + \norm{\partial_ty +\div\ssigma}{L^2(Q)}^2.
\end{align*}
The least-squares operators are given by
\begin{align*}
  \LL\yy &= \left(\partial_t y -\div\ssigma, \nabla y -\ssigma, y(0)\right), \quad
  \LL^\star\pp = \left(-\partial_t p-\div\xxi, \nabla p -\xxi, p(T)\right). 
\end{align*}
In~\cite[Theorem~2.3 and Remark~2.6]{GS21} it is shown that $\LL\colon Y\to Z$ is an isomorphism. The same argumentation proves that $\LL^\star$ is an isomorphism.
To see~\eqref{ass:I} we employ integration by parts in space and in time to get for all $(y,\ssigma)\in Y$, $(p,\xxi)\in Y^\star$
\begin{align*}
  \ip{\LL(y,\ssigma)}{\II^\star(p,\xxi)}_Z &= \ip{\partial_t y -\div\ssigma}{p}_{L^2(Q)} + \ip{\nabla y -\ssigma}{\xxi}_{L^2(Q)} + \ip{y(0)}{p(0)}_{L^2(\Omega)}
  \\
  &= \ip{\ssigma}{\nabla p-\xxi}_{L^2(Q)} + \ip{y}{-\partial_t p-\div\xxi}_{L^2(Q)} + \ip{y(T)}{p(T)}_{L^2(\Omega)}
  \\
  &= \ip{\II(y,\ssigma)}{\LL^\star(p,\xxi)}_Z.
\end{align*}


\section{Numerical experiments}\label{sec:num}
In this section we present numerical results for some of the examples described in Section~\ref{sec:examples}. 
We consider optimal control problems subject to the Poisson equation (Subsection~\ref{sec:num:poisson:unconstrained} for an unconstrained problem and Subsections~\ref{sec:num:poisson:constrained},\ref{sec:num:lshape} for a constrained problem), the Stokes equations (Subsection~\ref{sec:num:stokes}) and the heat equation (Subsection~\ref{sec:num:heat}). 

All programs have been implemented in \textsc{MATLAB} version 2022a on a \textsc{Linux} machine with \textsc{AMD Ryzen 1700+} processor and 32GB RAM. 

\subsection{Discretization}
For all problems except the heat equation we consider a two-dimensional domain $\Omega$ with polygonal boundary. 
For the heat equation we consider the time interval $J=(0,T)$ and a one-dimensional spatial domain $\omega$ (i.e., interval), hence, 
the space-time domain $Q = J\times \omega$ is a two-dimensional domain with polygonal boundary. For simplicity we write $\Omega$ for spatial and space-time domains in the definitions below. 

We consider a regular triangulation $\TT$ of $\Omega$ into open, non-empty triangles, 
\begin{align*}
  \overline\Omega = \bigcup_{T\in\TT} \overline{T}.
\end{align*}
The maximum element diameter is denoted by $h$. 
The following finite element spaces are used:
\begin{align*}
  \PP^p(\TT) &:= \set{v\in L^2(\Omega)}{v|_T \text{ is a polynomial of degree }\leq p, \,\text{for all }T\in\TT}, \\
  \cS^1(\TT) &:= \PP^1(\TT)\cap H^1(\Omega), \quad \cS_0^1(\TT) := \PP^{1}(\TT)\cap H_0^1(\Omega), \\
  \RT^0(\TT) &:= \set{\ttau\in \Hdivset\Omega}{\ttau|_T(x,y) = (\alpha,\beta)^\top + \gamma(x,y)^\top, \, \alpha,\beta,\gamma\in\R, \,\text{for all }T\in\TT}.
\end{align*}

We consider box constraints, i.e., $\Xad = \set{v\in L^2(\Omega;\R)}{a\leq v \leq b \text{ a.e. in }\Omega}$ with $a<b$ and use the discrete space
\begin{align*}
  X_{h,\mathrm{ad}} := \Xad\cap \PP^0(\TT). 
\end{align*}
By equipping $\PP^0(\TT)$ with the canonical basis of characteristic functions, the constraints in $X_{h,\mathrm{ad}}$ are pointwise constraints and we use the active set strategy proposed in~\cite[Algorithm~A2]{KKT03} to solve the discretized variational inequalities. 

We employ a standard adaptive finite element loop consisting of the four basic steps: \textsc{Solve}, \textsc{Estimate}, \textsc{Mark}, and \textsc{Refine}. 
For the estimation step we use the a posteriori error estimators proposed in Section~\ref{sec:lsfem:constrained} for constrained problems and in Section~\ref{sec:lsfem:unconstrained} for unconstrained problems. Note that both estimators can be written as a sum of local contributions, i.e., 
\begin{align*}
  \xi^2 = \sum_{T\in\TT} \xi(T)^2,
\end{align*}
where $\xi$ denotes one of the proposed estimators. 
As marking strategy in the example from Section~\ref{sec:num:lshape} we use the bulk criterion: Find a (minimal) set $\mathcal{M}\subseteq \TT$ such that
\begin{align*}
  \frac14 \xi^2 \leq \sum_{T\in\mathcal{M}} \xi(T)^2. 
\end{align*}
The refinement step is realized employing the newest-vertex bisection algorithm. 
Uniform refinements correspond to the case where all elements in the mesh are marked for refinement.

Let us note that the bilinear form $a(\cdot;\cdot)$ is coercive provided that the scaling parameters $\alpha$, $\beta$ are chosen sufficiently large, see Theorem~\ref{thm:constrained}.
In the numerical experiments we have chosen $\alpha=1$, $\beta=\lambda^{-1}$ and found that the results do not significantly change for different choices of this parameter provided that $\alpha\eqsim 1$, $\beta\eqsim \lambda^{-1}$.

\subsection{Unconstrained problem subject to Poisson equation}\label{sec:num:poisson:unconstrained}
Let $\Omega = (0,1)^2$. We consider the optimal control problem
\begin{align}\label{eq:optCtr:poisson}
\begin{split}
  \min_{u\in L^2(\Omega)} \norm{y-z_d}{L^2(\Omega)}^2 + &\lambda \norm{u}{L^2(\Omega)}^2 \quad\text{subject to}\\
  -\Delta y &= u + f \quad\text{in }\Omega, \\
  y|_{\partial\Omega} &= 0. 
\end{split}
\end{align}
In view of the abstract framework we have that $X = \Xad = L^2(\Omega;\R)$ and thus an unconstrained optimal control problem.
For this problem we consider the manufactured solution
\begin{align}\label{eq:optCtr:poisson:sol}
\begin{split}
  y(x_1,x_2) &= \sin(\pi x_1)\sin(\pi x_2), \\
  p(x_1,x_2) &= x_1(1-x_1)x_2(1-x_2),
\end{split}
\end{align}
and $u$ resp. $f$, $z_d$ are then computed by
\begin{align*}
  u &= -\frac1\lambda p, \\
  f &= -\Delta y -u, \\
  z_d &= \Delta p + y. 
\end{align*}

Since this is an unconstrained control problem we use the method described in Section~\ref{sec:lsfem:unconstrained} with the lowest-order finite element space
\begin{align*}
  Y_h \times Y_h^\star = (\cS_0^1(\TT)\times \RT^0(\TT))\times(\cS_0^1(\TT)\times \RT^0(\TT)),
\end{align*}
i.e., we compute minimizer $(\yy_h,\pp_h)$ of the functional $G(\cdot;f,z_d)$ over the space $Y_h\times Y_h^\star$. Since the solution is smooth we expect that
\begin{align*}
  \norm{\yy-\yy_h}Y + \norm{\pp-\pp_h}{Y^\star} = \OO(h).
\end{align*}
This is observed in the left plot of Figure~\ref{fig:poisson} where we visualize the errors $\norm{\yy-\yy_h}Y$, $\norm{\pp-\pp_h}{Y}$ and the a posteriori error estimator $G(\yy_h,\pp_h;f,z_d)$ for $\lambda=10^{-2}$.
The bottom row of Figure~\ref{fig:poisson} shows the effectivity index
\begin{align*}
  \mathrm{eff_{uc}}:=\frac{\sqrt{G(\yy_h,\pp_h;f,z_d)}}{\norm{(\yy-\yy_h,\pp-\pp_h)}{Y\times Y^\star}}
\end{align*}
for different values of $\lambda$.
Here, we use uniform mesh refinement.

\begin{figure}
  \begin{center}
    \begin{tikzpicture}
  \begin{groupplot}[group style={group size= 2 by 1},width=0.5\textwidth,cycle list/Dark2-6,
                      cycle multiindex* list={
                          mark list*\nextlist
                          Dark2-6\nextlist},
                      every axis plot/.append style={ultra thick},
                      grid=major,
                      xlabel={degrees of freedom},
                    ymode=log,
                    xmode=log]
        \nextgroupplot[title=Unconstrained,
          legend entries={\small {$G(\yy,\pp;f,z_d)$},\small $\|\yy-\yy_h\|_Y$,\small $\|\pp-\pp_h\|_Y$},
                      legend pos=north east]
                \addplot table [x=dofLSQ,y=estLSQ] {data/ExamplePoissonUnconstrained.dat};
                \addplot table [x=dofLSQ,y=errState] {data/ExamplePoissonUnconstrained.dat};
                \addplot table [x=dofLSQ,y=errAdjoint] {data/ExamplePoissonUnconstrained.dat};
                \addplot [black,dotted,mark=none] table [x=dofLSQ,y expr={0.6*sqrt(\thisrowno{1})^(-1)}] {data/ExamplePoissonUnconstrained.dat};
        \nextgroupplot[title=Constrained,
          legend entries={\small {$\eta$},\small $\|\yy-\yy_h\|_Y$,\small $\|\pp-\pp_h\|_Y$,\small $\|u-u_h\|$},
                      legend pos=north east]
                \addplot table [x=dofLSQ,y=estLSQ] {data/ExamplePoissonConstrained.dat};
                \addplot table [x=dofLSQ,y=errState] {data/ExamplePoissonConstrained.dat};
                \addplot table [x=dofLSQ,y=errAdjoint] {data/ExamplePoissonConstrained.dat};
                \addplot table [x=dofLSQ,y=errU] {data/ExamplePoissonConstrained.dat};
                \addplot [black,dotted,mark=none] table [x=dofLSQ,y expr={0.6*sqrt(\thisrowno{1})^(-1)}] {data/ExamplePoissonConstrained.dat};
    \end{groupplot}
%
\end{tikzpicture}
    \begin{tikzpicture}
  \begin{groupplot}[group style={group size= 2 by 1},width=0.5\textwidth,cycle list/Dark2-6,
                      cycle multiindex* list={
                          mark list*\nextlist
                          Dark2-6\nextlist},
                      every axis plot/.append style={ultra thick},
                      grid=major,
                      xlabel={degrees of freedom},
                    xmode=log]
          \nextgroupplot[title=Unconstrained,ymin = 0, ymax=2,
          legend entries={\small $\lambda=10^{-1}$,\small $\lambda=10^{-2}$,\small $\lambda=10^{-3}$,\small $\lambda=10^{-4}$,\small $\lambda=10^{-5}$,\small $\lambda=10^{-6}$},
                      legend pos=north west]
                \addplot table [x=dofLSQ,y=effInd] {data/ExamplePoissonUnconstrained_1e-01.dat};
                \addplot table [x=dofLSQ,y=effInd] {data/ExamplePoissonUnconstrained_1e-02.dat};
                \addplot table [x=dofLSQ,y=effInd] {data/ExamplePoissonUnconstrained_1e-03.dat};
                \addplot table [x=dofLSQ,y=effInd] {data/ExamplePoissonUnconstrained_1e-04.dat};
                \addplot table [x=dofLSQ,y=effInd] {data/ExamplePoissonUnconstrained_1e-05.dat};
                \addplot table [x=dofLSQ,y=effInd] {data/ExamplePoissonUnconstrained_1e-06.dat};
        \nextgroupplot[title=Constrained,ymin = 0, ymax=8,
                legend entries={\small $\lambda=10^{-1}$,\small $\lambda=10^{-2}$,\small $\lambda=10^{-3}$,\small $\lambda=10^{-4}$,\small $\lambda=10^{-5}$,\small $\lambda=10^{-6}$},
                      legend pos=north east]
                \addplot table [x=dofLSQ,y=effInd] {data/ExamplePoissonConstrained_1e-01.dat};
                \addplot table [x=dofLSQ,y=effInd] {data/ExamplePoissonConstrained_1e-02.dat};
                \addplot table [x=dofLSQ,y=effInd] {data/ExamplePoissonConstrained_1e-03.dat};
                \addplot table [x=dofLSQ,y=effInd] {data/ExamplePoissonConstrained_1e-04.dat};
                \addplot table [x=dofLSQ,y=effInd] {data/ExamplePoissonConstrained_1e-05.dat};
                \addplot table [x=dofLSQ,y=effInd] {data/ExamplePoissonConstrained_1e-06.dat};
    \end{groupplot}
\end{tikzpicture}
  \end{center}
  \caption{Top row: Errors and error estimator for the numerical experiments from Section~\ref{sec:num:poisson:unconstrained} (left plot) resp. Section~\ref{sec:num:poisson:constrained} (right plot) with $\lambda = 10^{-2}$. 
  The black dashed lines indicate the convergence rate $\OO(h)$.
  Bottom row: Effectivity indices for different values of $\lambda$.}
  \label{fig:poisson}
\end{figure}
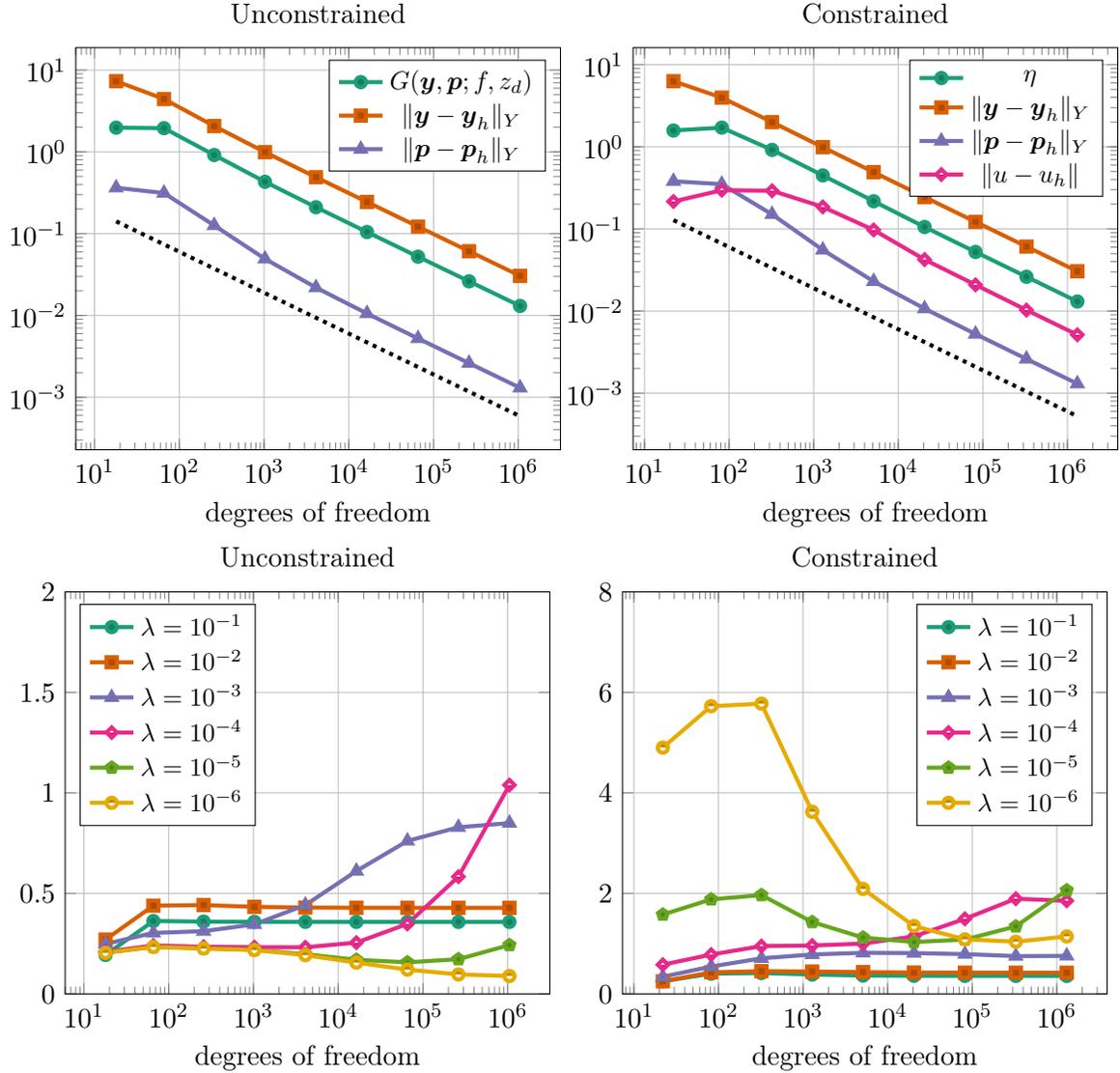

\subsection{Constrained problem subject to Poisson equation}\label{sec:num:poisson:constrained}
In this section we consider the control problem~\eqref{eq:optCtr:poisson} and the same setting again but replace the control space $\Xad=L^2(\Omega;\R)$ by 
\begin{align*}
  \Xad = \set{v\in L^2(\Omega;\R)}{-1\leq v \leq 0 \text{ a.e. in }\Omega}.
\end{align*}
We prescribe the solutions~\eqref{eq:optCtr:poisson:sol} and define
\begin{align*}
  u = \Pi_{\Xad} (-\lambda^{-1} p),
\end{align*}
and compute the data $f$, $z_d$ by
\begin{align*}
  f &= -\Delta y -u, \\
  z_d &= \Delta p + y. 
\end{align*}
We use the proposed method described in Section~\ref{sec:lsfem:constrained}, i.e., we solve the variational inequality~\eqref{eq:varineq:general} with
the lowest-order finite element space
\begin{align*}
  H = \Xad\cap \PP^0(\TT)\times (\cS_0^1(\TT) \times \RT^0(\TT)) \times (\cS_0^1(\TT) \times \RT^0(\TT)).
\end{align*}
The solution pair $(y,p)$ is smooth and $u \in H^1(\Omega)$. Let $\Pi_h^0$ denote the $L^2(\Omega)$ orthogonal projection onto piecewise constants. 
Using Corollary~\ref{cor:apriori} we infer that
\begin{align*}
  &\norm{u-u_h}{L^2(\Omega)}^2 + \norm{\yy-\yy_h}{Y}^2 + \norm{\pp-\pp_h}{Y}^2 
  \\
  &\qquad \lesssim \norm{u-\Pi_h^0 u}{L^2(\Omega)}^2 + \norm{\yy-\zz_h}Y^2 + \norm{\pp-\qq_h}Y^2 + |\ip{p+\lambda u}{\Pi_h^0u-u}_{L^2(\Omega)}|
\end{align*}
for $(\zz_h,\qq_h)\in Y_h\times Y_h$, which is possible since $\Pi_h^0u \in[a,b]$ if $a\leq u \leq b$ a.e., thus, $\Pi_h^0 u\in X_{h,\mathrm{ad}}$.
Employing $\norm{(1-\Pi_h^0)u}{L^2(\Omega)} \lesssim h\norm{\nabla u}{L^2(\Omega)}$ we infer that 
\begin{align*}
  &\norm{(1-\Pi_h^0)u}{L^2(\Omega)}^2 + |\ip{p+\lambda u}{(1-\Pi_h^0)u}_{L^2(\Omega)}| 
  \\
  &\qquad= \norm{(1-\Pi_h^0)u}{L^2(\Omega)}^2 + |\ip{(1-\Pi_h^0)(p+\lambda u)}{(1-\Pi_h^0)u}_{L^2(\Omega)}|
  \\ &\qquad\lesssim h^2\norm{\nabla u}{L^2(\Omega)}^2 + h\norm{\nabla(p+\lambda u)}{L^2(\Omega)}h \norm{\nabla u}{L^2(\Omega)} = \OO(h^2).
\end{align*}
Putting all together we obtain with standard approximation results that
\begin{align*}
  \norm{u-u_h}{L^2(\Omega)} + \norm{\yy-\yy_h}{Y} + \norm{\pp-\pp_h}{Y} = \OO(h).
\end{align*}
This rate is also observed in our numerical experiment for the moderate value $\lambda=10^{-2}$, see the top row of Figure~\ref{fig:poisson} (right plot).
In the bottom row we also show the effectivity index
\begin{align*}
  \mathrm{eff_c}:=\frac{\eta}{\norm{(\uu-\uu_h,\yy-\yy_h,\pp-\pp_h)}{X\times Y\times Y^\star}}
\end{align*}
for different values of $\lambda$.
Here, we use uniform mesh refinement.

\subsection{Constrained problem with singular solution}\label{sec:num:lshape}
We consider the control problem~\eqref{eq:optCtr:poisson} with $\lambda=1$ and the L-shaped domain $\Omega = (-1,1)^2\setminus [-1,0]^2$.
The set of admissible controls is given by
\begin{align*}
  \Xad = \set{v\in L^2(\Omega)}{0.1\leq v \leq 0.12 \text{ a.e. in }\Omega}
\end{align*}
and as data we use
\begin{align*}
  f(x_1,x_2) &= 0, \, z_d(x_1,x_2) = 1, \quad (x_1,x_2)\in \Omega.
\end{align*}
For this setup we do not know an explicit representation of the solution $(u,y,p)$ but stress that reduced regularity is expected due to the non-convexity of the domain.
Figure~\ref{fig:lshape} shows that the error estimator $\eta$ seems to asymptotically converge at $\OO(h^{2/3})$ for a sequence of uniform meshes. 
Using adaptively refined meshes we observe that optimal rates are reestablished, i.e., the error estimator converges at $\OO(N^{-1/2})$ with $N=\dim(X_h\times Y_h\times Y_h)$. 
The right plot in top row of Figure~\ref{fig:lshape} visualizes the solution $u_h \in X_{h,\mathrm{ad}}$ on a mesh with $\#\TT=720$ triangles obtained from the adaptive loop. 
Strong mesh refinements towards the re-entrant corner are observed (Figure~\ref{fig:lshape} bottom row), indicating that the adaptive algorithm detects singularities.

\begin{figure}
  \begin{center}
    \begin{tikzpicture}
\begin{loglogaxis}[
    width=0.49\textwidth,
cycle list/Dark2-6,
cycle multiindex* list={
mark list*\nextlist
Dark2-6\nextlist},
every axis plot/.append style={ultra thick},
xlabel={degrees of freedom},
grid=major,
legend entries={\small $\eta$ unif.,\small $\eta$ adap.},
legend pos=south west,
]

\addplot table [x=dofLSQ,y=estLSQ] {data/ExampleLshapeUnif.dat};
\addplot table [x=dofLSQ,y=estLSQ] {data/ExampleLshapeAdap.dat};
\addplot [black,dotted,mark=none] table [x=dofLSQ,y expr={1*sqrt(\thisrowno{1})^(-2/3)}] {data/ExampleLshapeUnif.dat};
\addplot [black,dotted,mark=none] table [x=dofLSQ,y expr={3*sqrt(\thisrowno{1})^(-1)}] {data/ExampleLshapeAdap.dat};
\end{loglogaxis}
\end{tikzpicture}
\begin{tikzpicture}
  \begin{axis}[
width=0.49\textwidth,
view={-45}{60},
xlabel={$x$},
ylabel={$y$},
hide axis,
    colorbar,
]
\addplot3[patch,hide axis,line width=0.2pt,faceted color=black] table{data/LshapeSol.dat};
\end{axis}
\end{tikzpicture}
\begin{tikzpicture}
\begin{axis}[hide axis,
width=0.5\textwidth,
    axis equal,
]

\addplot[patch,color=white,
faceted color = black, line width = 0.5pt,
patch table ={data/LshapeEle2.dat}] file{data/LshapeCoo2.dat};
\end{axis}
\end{tikzpicture}
\begin{tikzpicture}
\begin{axis}[hide axis,
width=0.5\textwidth,
    axis equal,
]

\addplot[patch,color=white,
faceted color = black, line width = 0.5pt,
patch table ={data/LshapeEle.dat}] file{data/LshapeCoo.dat};
\end{axis}
\end{tikzpicture}
  \end{center}
  \caption{Top row: The left plot shows the error estimator $\eta$ on a sequence of uniformly refined meshes and a sequence of locally refined meshes generated by the adaptive algorithm. The black dashed lines indicate $\OO(N^{-1/3})$ resp. $\OO(N^{-1/2}$) where $N=\dim(X_{h,\mathrm{ad}}\times Y_h\times Y_h)$. 
    The right plot shows the discrete control $u_h$ on a mesh with $\#\TT =720$ triangles.
Bottom row: Two meshes generated by the adaptive algorithm with $\#\TT=513$ resp. $\#\TT=720$ elements.}\label{fig:lshape}
\end{figure}
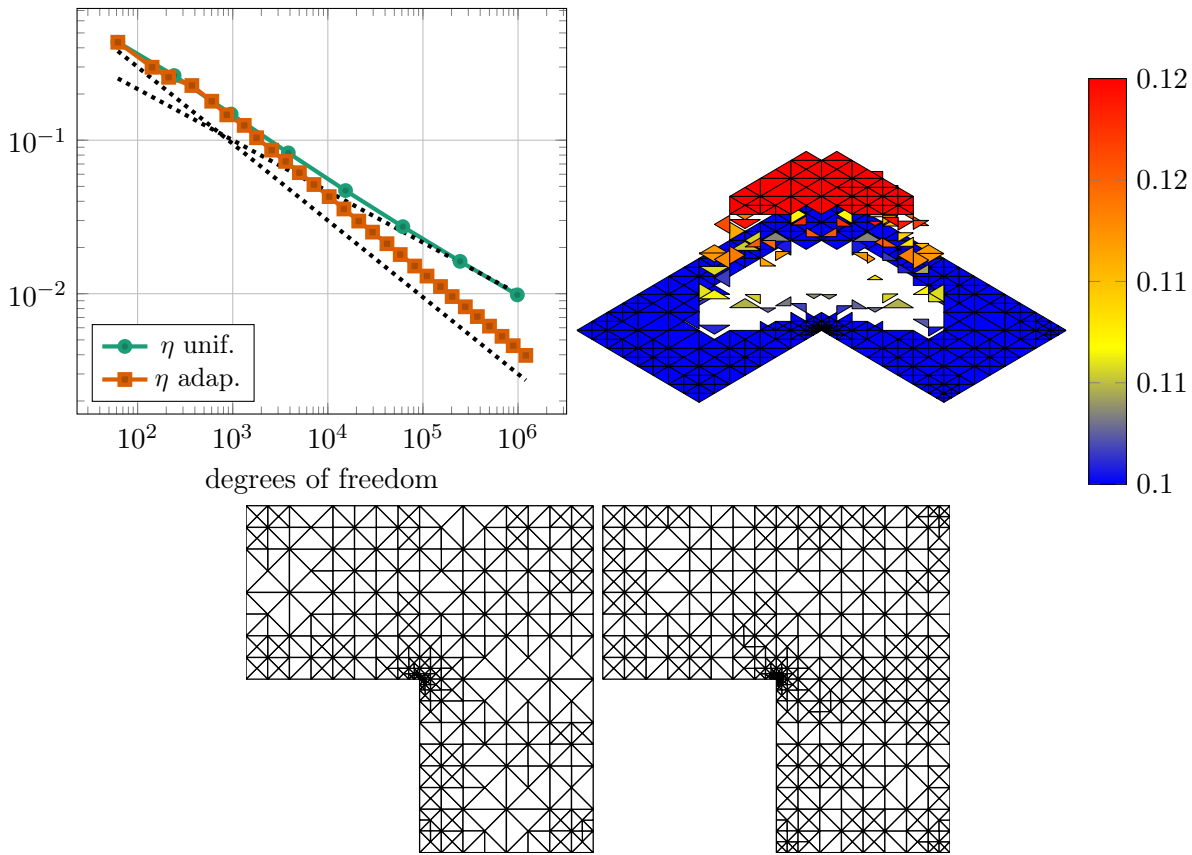

\subsection{Unconstrained problem subject to Stokes equation}\label{sec:num:stokes}
Let $\Omega =(0,1)^2$. We consider the optimal control problem~\eqref{eq:stokes} with $X = \Xad$, i.e., the unconstrained control problem. Our manufactured solution is given by
\begin{align*}
  \yy(x,y) = \pp(x,y) = \curl \big( x(1-x)y(1-y)\big)^2
\end{align*}
and the data $\ff = -\boldsymbol\Delta\yy+\lambda^{-1}\pp$, $\zz_d = \boldsymbol\Delta\pp+\yy$ is computed thereof. 
Since this is an unconstrained control problem we employ the method described in Section~\ref{sec:lsfem:unconstrained}, see also Section~\ref{sec:num:poisson:unconstrained} for a similar numerical experiment for the Poisson problem. 
The solution is smooth so that we expect optimal convergence of the lowest-order scheme which is also observed in Figure~\ref{fig:stokes} for different values of $\lambda$. Here, we use uniform mesh refinement.
Figure~\ref{fig:stokes} also shows the effectivity index $\mathrm{eff_{uc}}$.
One observes that as $\lambda$ gets smaller the error in $\yy$ has a pre-asymptotic range where it converges at a higher rate.
This explains that for $\lambda=10^{-3}$ the effectivity index grows.

\begin{figure}
  \begin{center}
    \begin{tikzpicture}
\begin{loglogaxis}[
    title={\small error and estimator $\lambda=10^{-1}$},
    width=0.49\textwidth,
cycle list/Dark2-6,
cycle multiindex* list={
mark list*\nextlist
Dark2-6\nextlist},
every axis plot/.append style={ultra thick},
xlabel={degrees of freedom},
grid=major,
legend entries={\small {$G(\yy_h,\pp_h;\ff,\zz_d)$},\small $\|\yy-\yy_h\|_Y$,\small $\|\pp-\pp_h\|_Y$},
legend pos=north east
]

\addplot table [x=dofLSQ,y=estLSQ] {data/ExampleStokesUnconstrained_1e-01.dat};
\addplot table [x=dofLSQ,y=errState] {data/ExampleStokesUnconstrained_1e-01.dat};
\addplot table [x=dofLSQ,y=errAdjoint] {data/ExampleStokesUnconstrained_1e-01.dat};
\addplot [black,dotted,mark=none] table [x=dofLSQ,y expr={3*sqrt(\thisrowno{1})^(-1)}] {data/ExampleStokesUnconstrained_1e-01.dat};
\end{loglogaxis}
\end{tikzpicture}
\begin{tikzpicture}
\begin{loglogaxis}[
    title={\small error and estimator $\lambda=10^{-2}$},
    width=0.49\textwidth,
cycle list/Dark2-6,
cycle multiindex* list={
mark list*\nextlist
Dark2-6\nextlist},
every axis plot/.append style={ultra thick},
xlabel={degrees of freedom},
grid=major,
legend entries={\small {$G(\yy_h,\pp_h;\ff,\zz_d)$},\small $\|\yy-\yy_h\|_Y$,\small $\|\pp-\pp_h\|_Y$},
legend pos=north east
]

\addplot table [x=dofLSQ,y=estLSQ] {data/ExampleStokesUnconstrained_1e-02.dat};
\addplot table [x=dofLSQ,y=errState] {data/ExampleStokesUnconstrained_1e-02.dat};
\addplot table [x=dofLSQ,y=errAdjoint] {data/ExampleStokesUnconstrained_1e-02.dat};
\addplot [black,dotted,mark=none] table [x=dofLSQ,y expr={3*sqrt(\thisrowno{1})^(-1)}] {data/ExampleStokesUnconstrained_1e-02.dat};
\end{loglogaxis}
\end{tikzpicture}
\begin{tikzpicture}
\begin{loglogaxis}[
    title={\small error and estimator $\lambda=10^{-3}$},
    width=0.49\textwidth,
cycle list/Dark2-6,
cycle multiindex* list={
mark list*\nextlist
Dark2-6\nextlist},
every axis plot/.append style={ultra thick},
xlabel={degrees of freedom},
grid=major,
legend entries={\small {$G(\yy_h,\pp_h;\ff,\zz_d)$},\small $\|\yy-\yy_h\|_Y$,\small $\|\pp-\pp_h\|_Y$},
legend pos=north east
]

\addplot table [x=dofLSQ,y=estLSQ] {data/ExampleStokesUnconstrained_1e-03.dat};
\addplot table [x=dofLSQ,y=errState] {data/ExampleStokesUnconstrained_1e-03.dat};
\addplot table [x=dofLSQ,y=errAdjoint] {data/ExampleStokesUnconstrained_1e-03.dat};
\addplot [black,dotted,mark=none] table [x=dofLSQ,y expr={3*sqrt(\thisrowno{1})^(-1)}] {data/ExampleStokesUnconstrained_1e-03.dat};
\end{loglogaxis}
\end{tikzpicture}
\begin{tikzpicture}
  \begin{axis}[xmode=log,ymin=0,ymax=5,
      title={\small effectivity index},
    width=0.49\textwidth,
cycle list/Dark2-6,
cycle multiindex* list={
mark list*\nextlist
Dark2-6\nextlist},
every axis plot/.append style={ultra thick},
xlabel={degrees of freedom},
grid=major,
legend entries={\small $\lambda=10^{-1}$,\small $\lambda=10^{-2}$,\small $\lambda=10^{-3}$},
legend pos=north west
]

\addplot table [x=dofLSQ,y=eff] {data/ExampleStokesUnconstrained_1e-01.dat};
\addplot table [x=dofLSQ,y=eff] {data/ExampleStokesUnconstrained_1e-02.dat};
\addplot table [x=dofLSQ,y=eff] {data/ExampleStokesUnconstrained_1e-03.dat};
\end{axis}
\end{tikzpicture}
  \end{center}
  \caption{Error estimator and errors for the optimal control problem~\eqref{eq:stokes} considered in Section~\ref{sec:num:stokes}. The dashed black line indicates $\OO(h)$. The right plot in the bottom row shows the effectivity index $\mathrm{eff_{uc}}$.}\label{fig:stokes}
\end{figure}
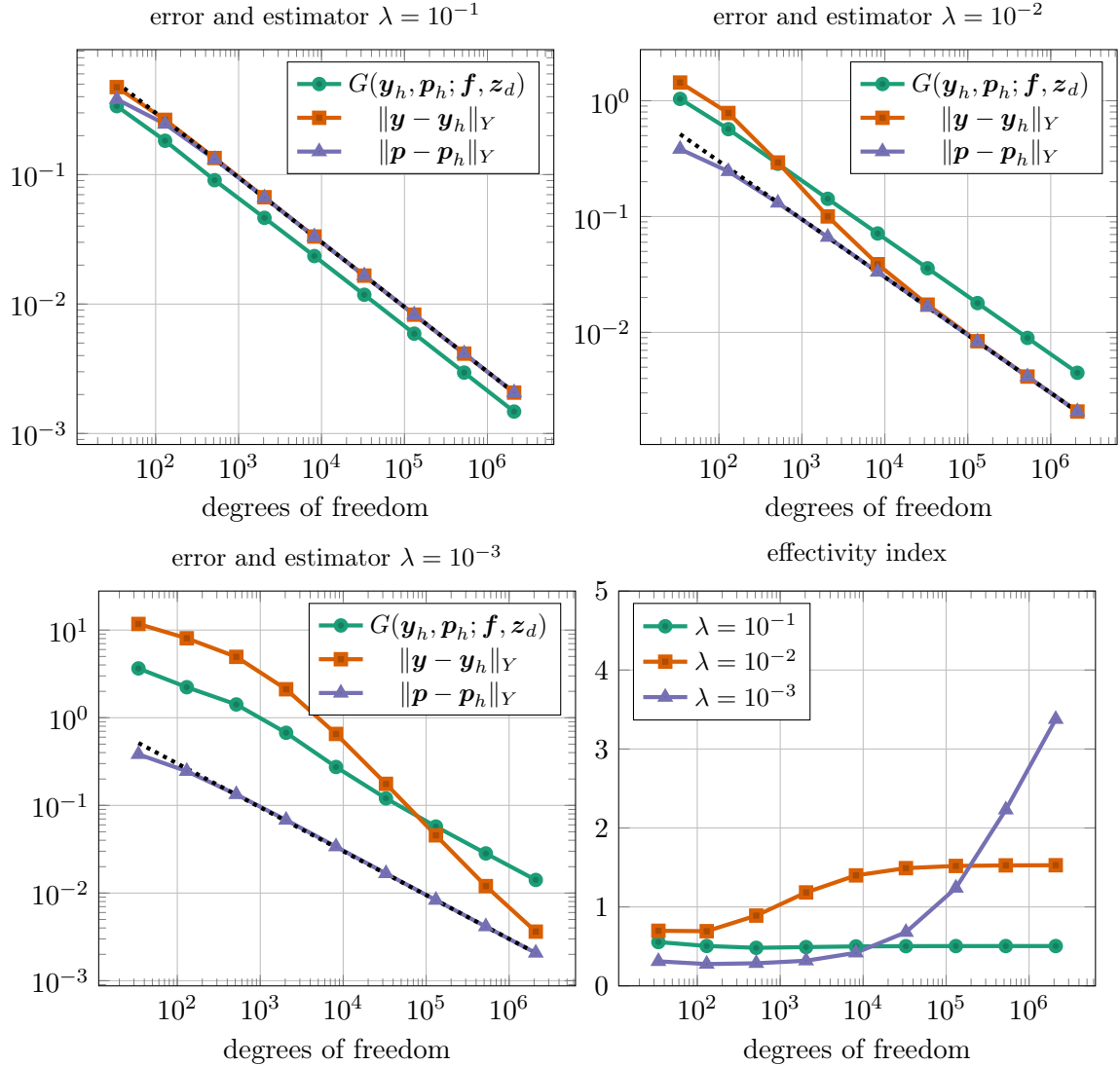

\subsection{Constrained problem subject to heat equation}\label{sec:num:heat}
Let $J = (0,1) = \Omega$, $Q=J\times \Omega$. 
We consider the optimal control problem~\eqref{eq:heat} with 
\begin{align*}
  \Xad = \set{(v,v_0)\in L^2(Q)\times L^2(\Omega)}{-1\leq v\leq 0, \, -1\leq v_0\leq 0}. 
\end{align*}
We use the manufactured solution
\begin{align*}
  y(x,t) &= t\sin(\pi x), \\
  p(x,t) &= (1-t)x(1-x). 
\end{align*}
Then, $(u,u_0) = \Pi_{\Xad}(-\lambda^{-1}p,-\lambda^{-1}p_0)$ and the data $f$, $z_d$, $z_{d,T}$ is computed by
\begin{align*}
  f &= \partial_t y - \partial_{xx} y-u, \\
  z_d &= \partial_t p + \partial_{xx} p+y, \\
  z_{d,T} &= -p(T)+y(T).
\end{align*}
Since the solution $(y,p)$ is smooth and $u\in H^1(Q)$, $u_0\in H^1(\Omega)$ one expects that
\begin{align*}
  \norm{u-u_h}{L^2(Q)}+\norm{u_0-u_h(0)}{L^2(\Omega)} + \norm{\yy-\yy_h}Y + \norm{\pp-\pp_h}{Y^\star}=\OO(h),
\end{align*}
see Theorem~\ref{thm:apriori} by using a similar argumentation as in Section~\ref{sec:num:poisson:constrained}. 
Figure~\ref{fig:heat:smooth} shows the a posteriori estimators and errors for this problem for different values of $\lambda$.
We observe the optimal rate $\OO(h) = \OO(N^{-1/2})$ where $N = \dim(X_{h}\times Y_h\times Y_h^\star)$ on a sequence of uniformly refined meshes.
Furthermore, Figure~\ref{fig:heat:smooth} also shows the effectivity index $\mathrm{eff_c}$.
We note that for $\lambda\leq 10^{-3}$ the employed solver (active set strategy) did not converge within reasonable time. This requires some further analysis and possible fine-tuning of parameters.

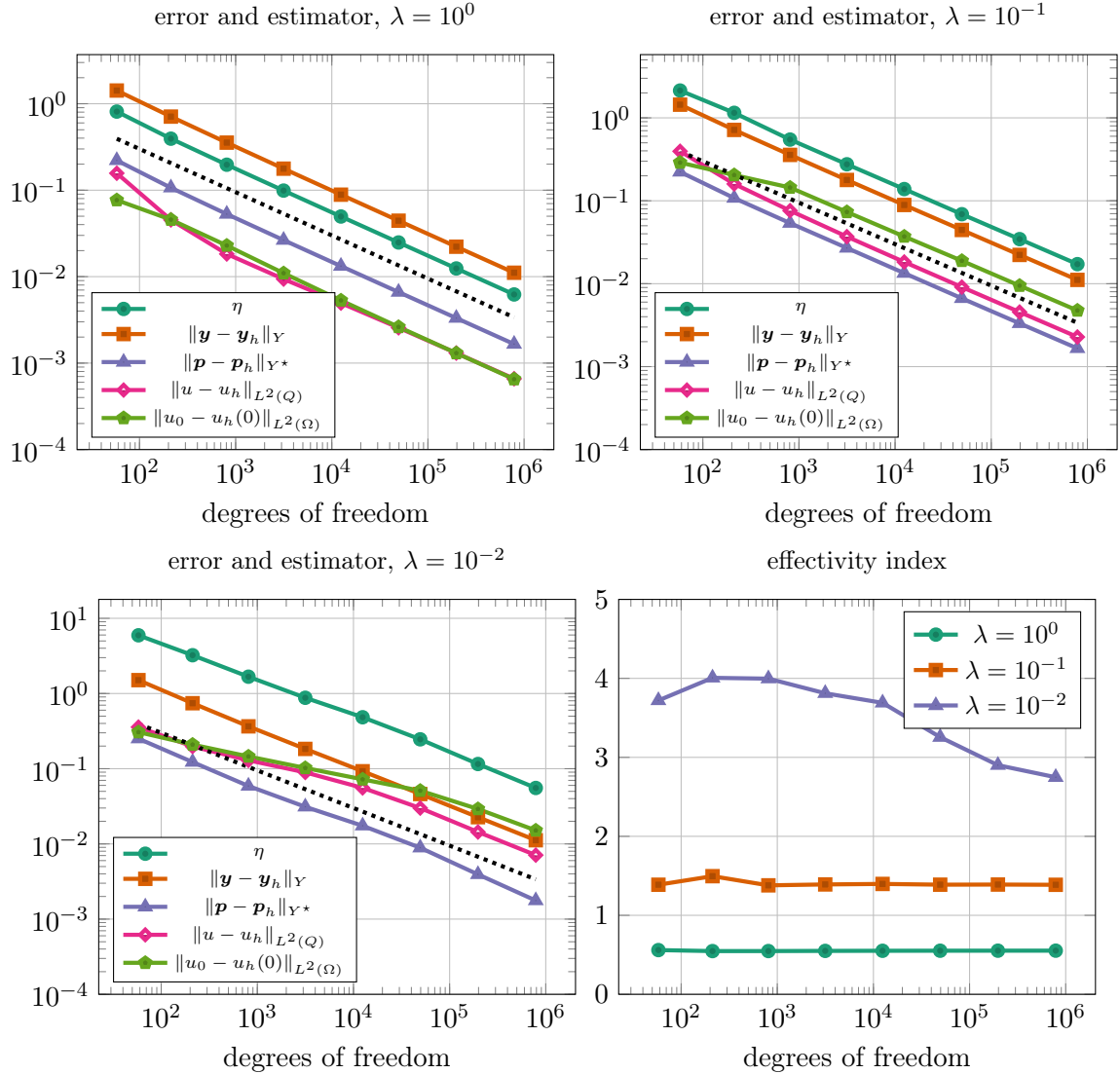
\begin{figure}
  \begin{center}
    \begin{tikzpicture}
\begin{loglogaxis}[
    title={\small error and estimator, $\lambda = 10^{0}$},
    width=0.49\textwidth,
cycle list/Dark2-6,
cycle multiindex* list={
mark list*\nextlist
Dark2-6\nextlist},
every axis plot/.append style={ultra thick},
xlabel={degrees of freedom},
grid=major,
legend entries={\tiny $\eta$,\tiny $\|\yy-\yy_h\|_Y$,\tiny $\|\pp-\pp_h\|_{Y^\star}$,\tiny $\|u-u_h\|_{L^2(Q)}$,\tiny $\|u_0-u_h(0)\|_{L^2(\Omega)}$},
legend pos=south west,
ymin=10^(-4),
]

\addplot table [x=dofLSQ,y=estLSQ] {data/ExampleHeatSmooth_ell1e+00.dat};
\addplot table [x=dofLSQ,y=errState] {data/ExampleHeatSmooth_ell1e+00.dat};
\addplot table [x=dofLSQ,y=errAdjoint] {data/ExampleHeatSmooth_ell1e+00.dat};
\addplot table [x=dofLSQ,y=errU] {data/ExampleHeatSmooth_ell1e+00.dat};
\addplot table [x=dofLSQ,y=errU0] {data/ExampleHeatSmooth_ell1e+00.dat};
\addplot [black,dotted,mark=none] table [x=dofLSQ,y expr={3*sqrt(\thisrowno{1})^(-1)}] {data/ExampleHeatSmooth_ell1e+00.dat};
\end{loglogaxis}
\end{tikzpicture}
\begin{tikzpicture}
\begin{loglogaxis}[
    title={\small error and estimator, $\lambda = 10^{-1}$},
    width=0.49\textwidth,
cycle list/Dark2-6,
cycle multiindex* list={
mark list*\nextlist
Dark2-6\nextlist},
every axis plot/.append style={ultra thick},
xlabel={degrees of freedom},
grid=major,
legend entries={\tiny $\eta$,\tiny $\|\yy-\yy_h\|_Y$,\tiny $\|\pp-\pp_h\|_{Y^\star}$,\tiny $\|u-u_h\|_{L^2(Q)}$,\tiny $\|u_0-u_h(0)\|_{L^2(\Omega)}$},
legend pos=south west,
ymin=10^(-4),
]

\addplot table [x=dofLSQ,y=estLSQ] {data/ExampleHeatSmooth_ell1e-01.dat};
\addplot table [x=dofLSQ,y=errState] {data/ExampleHeatSmooth_ell1e-01.dat};
\addplot table [x=dofLSQ,y=errAdjoint] {data/ExampleHeatSmooth_ell1e-01.dat};
\addplot table [x=dofLSQ,y=errU] {data/ExampleHeatSmooth_ell1e-01.dat};
\addplot table [x=dofLSQ,y=errU0] {data/ExampleHeatSmooth_ell1e-01.dat};
\addplot [black,dotted,mark=none] table [x=dofLSQ,y expr={3*sqrt(\thisrowno{1})^(-1)}] {data/ExampleHeatSmooth_ell1e-01.dat};
\end{loglogaxis}
\end{tikzpicture}
\begin{tikzpicture}
\begin{loglogaxis}[
    title={\small error and estimator, $\lambda = 10^{-2}$},
    width=0.49\textwidth,
cycle list/Dark2-6,
cycle multiindex* list={
mark list*\nextlist
Dark2-6\nextlist},
every axis plot/.append style={ultra thick},
xlabel={degrees of freedom},
grid=major,
legend entries={\tiny $\eta$,\tiny $\|\yy-\yy_h\|_Y$,\tiny $\|\pp-\pp_h\|_{Y^\star}$,\tiny $\|u-u_h\|_{L^2(Q)}$,\tiny $\|u_0-u_h(0)\|_{L^2(\Omega)}$},
legend pos=south west,
ymin=10^(-4),
]

\addplot table [x=dofLSQ,y=estLSQ] {data/ExampleHeatSmooth_ell1e-02.dat};
\addplot table [x=dofLSQ,y=errState] {data/ExampleHeatSmooth_ell1e-02.dat};
\addplot table [x=dofLSQ,y=errAdjoint] {data/ExampleHeatSmooth_ell1e-02.dat};
\addplot table [x=dofLSQ,y=errU] {data/ExampleHeatSmooth_ell1e-02.dat};
\addplot table [x=dofLSQ,y=errU0] {data/ExampleHeatSmooth_ell1e-02.dat};
\addplot [black,dotted,mark=none] table [x=dofLSQ,y expr={3*sqrt(\thisrowno{1})^(-1)}] {data/ExampleHeatSmooth_ell1e-02.dat};
\end{loglogaxis}
\end{tikzpicture}
\begin{tikzpicture}
  \begin{axis}[xmode=log,ymin=0,ymax=5,
      title={\small effectivity index},
    width=0.49\textwidth,
cycle list/Dark2-6,
cycle multiindex* list={
mark list*\nextlist
Dark2-6\nextlist},
every axis plot/.append style={ultra thick},
xlabel={degrees of freedom},
grid=major,
legend entries={\small $\lambda=10^{0}$,\small $\lambda=10^{-1}$,\small $\lambda=10^{-2}$},
legend pos=north east
]

\addplot table [x=dofLSQ,y=eff] {data/ExampleHeatSmooth_ell1e+00.dat};
\addplot table [x=dofLSQ,y=eff] {data/ExampleHeatSmooth_ell1e-01.dat};
\addplot table [x=dofLSQ,y=eff] {data/ExampleHeatSmooth_ell1e-02.dat};
\end{axis}
\end{tikzpicture}
  \end{center}
  \caption{Error estimator and errors for the optimal control problem~\eqref{eq:heat} considered in Section~\ref{sec:num:heat}. The dashed black line indicates $\OO(h)$. The right plot in the bottom row shows the effectivity index $\mathrm{eff_{c}}$.}\label{fig:heat:smooth}
\end{figure}


\bibliographystyle{abbrv}
\bibliography{literature}

\end{document}